\documentclass{amsart}

\makeindex

\usepackage[all, cmtip]{xy}
\usepackage{amssymb}
\usepackage{amsmath}
\usepackage{amsthm}
\usepackage{amscd}
\usepackage{amsfonts}
\usepackage{amssymb}
\usepackage{multirow}
\usepackage[usenames,dvipsnames,svgnames,table]{xcolor}
\usepackage{graphicx}
\usepackage{enumerate}

\usepackage[bookmarks=true,bookmarksnumbered=true,
 pdftitle={},
 pdfsubject={},
 pdfauthor={},
 pdfkeywords={TeX; dvipdfmx; hyperref; color;},
 colorlinks=true, linkcolor=Maroon, citecolor=ForestGreen]
 {hyperref}

\usepackage{longtable}
\usepackage{pgf,pgfarrows,pgfnodes,pgfautomata,pgfheaps,pgfshade}
\usepackage{tikz}

\usepackage{lmodern}

\makeatletter
\ifcase \@ptsize \relax% 10pt
  \newcommand{\miniscule}{\@setfontsize\miniscule{3}{7}}% \tiny: 5/6
    \newcommand{\stiny}{\@setfontsize\miniscule{5}{7}}% \tiny: 6/7
\or% 11pt
  \newcommand{\miniscule}{\@setfontsize\miniscule{3}{7}}% \tiny: 6/7
   \newcommand{\stiny}{\@setfontsize\miniscule{5}{7}}% \tiny: 6/7
\or% 12pt
  \newcommand{\miniscule}{\@setfontsize\miniscule{3}{7}}% \tiny: 6/7
    \newcommand{\stiny}{\@setfontsize\miniscule{5}{7}}% \tiny: 6/7
\fi
\makeatother

\theoremstyle{definition}

\theoremstyle{remark}

\theoremstyle{corollary}

\theoremstyle{theorem}

\theoremstyle{corollary}

\newtheorem{theorem}{Theorem}[section]
\newtheorem{lemma}[theorem]{Lemma}
\newtheorem{proposition}[theorem]{Proposition}
\newtheorem{conjecture}[theorem]{Conjecture}

\theoremstyle{corollary}
\newtheorem{corollary}[theorem]{Corollary}

\theoremstyle{definition}
\newtheorem{definition}[theorem]{Definition}

\theoremstyle{remark}
\newtheorem{remark}[theorem]{Remark}

\numberwithin{equation}{section}

\newcommand{\Z}{\mathbb{Z}}

\def\P{\mathbb{P}}

\def\Pic{\operatorname{Pic}}
\def\Proj{\operatorname{Proj}}

\def\Cox{\operatorname{Cox}}

\def\reg{\operatorname{reg}}
\def\pd{\operatorname{pd}}

\def\codim{\operatorname{codim}}

\def\Eff{\operatorname{Eff}}
\def\Nef{\operatorname{Nef}}
\def\index{\operatorname{index}}

\def\deg{\operatorname{deg}}
\def\rank{\operatorname{rank}}

\newcommand{\suchthat}{\;\ifnum\currentgrouptype=16 \middle\fi|\;}

\input xy
\xyoption{all}

\title[Hilbert functions of Cox rings of del Pezzo surfaces]{Hilbert functions of Cox rings of del Pezzo surfaces}
\thanks{The second author has been supported by the National Research Foundation in Korea (NRF-2014R1A1A2056432) and IBS-R003-D1}
\begin{document}

\author{Jinhyung Park}
\address{School of Mathematics, Korea Institute for Advanced Study, 85 Hoegiro, Dongdaemun-gu, Seoul 02455, Republic of Korea}
\email{parkjh13@kias.re.kr}

\author{Joonyeong Won}
\address{Center for Geometry and Physics, Institute for Basic Science (IBS) 77 Cheongam-ro, Nam-gu, Pohang, Gyeongbuk, 37673, Republic of Korea}
\email{leonwon@kias.re.kr}

%\thanks{}

\date{\today}
%\date{June 22, 2009 and, in revised form,}



\begin{abstract}
To study syzygies of the Cox rings of del Pezzo surfaces, we calculate important syzygetic invariants such as the Hilbert functions, the Green-Lazarsfeld indices, the projective dimensions, and the Castelnuovo-Mumford regularities. Using these computations as well as the natural multigrading structures by the Picard groups of del Pezzo surfaces and Weyl group actions on Picard lattices, we determine the Betti diagrams of the Cox rings of del Pezzo surfaces of degree at most four.
\end{abstract}

\maketitle

%\tableofcontents \setcounter{page}{1}

\section{Introduction}

In 1958, Nagata constructed the first counterexample to Hilbert's 14th problem, which asked whether the ring of invariance by a group action is finitely generated (\cite{N}).
Let $k$ be an algebraically closed field of characteristic zero.
Consider a group action on a polynomial ring $R_r:=k[x_1, \ldots, x_r, y_1, \ldots, y_r]$ by $r$-th power of the additive group $\mathbb{G}_a(k)^r$ given by $x_i \mapsto x_i, y_i \mapsto y_i + t_i x_i$ for $(t_1, \ldots, t_r) \in \mathbb{G}_a(k)^r$. Take a general linear subspace $G_n \subset \mathbb{G}_a(k)^r$ of dimension $n$ with the induced group structure. There is an induced group action of $G_n$ on $R_r$. Note that
$$
R_r^{G_n} \simeq \Cox(X_r^{r-n-1})
$$
where $X_r^{r-n-1}$ is the blow-up of $\P^{r-n-1}$ at $r$ general points (see \cite{Muk1}). The Cox ring was first introduced by Hu and Keel (\cite{HK}) based on Cox's construction of toric varieties (\cite{C}).
The Cox ring of a smooth projective variety $X$ with a finitely generated Picard group is defined as
$$
\Cox(X):=\bigoplus_{L \in \Pic(X)} H^0(X, L).
$$
Their explicit description captures much of important geometric and arithmetic properties (see e.g., \cite{HK} and \cite{CTS}). Nagata, in fact, proved that $\Cox(X_r^{r-n-1})$ is not finitely generated when $r=13, n=10$. In general, $\Cox(X_r^{r-n-1})$ is not finitely generated if and only if $\frac{1}{2} + \frac{1}{r-n} + \frac{1}{n} \leq 1$ (see \cite{CT}, \cite{Muk1}, \cite{Muk2}). In particular  $\Cox(X_r^2)$ is finitely generated when $r \leq 8$. It is well known that every del Pezzo surface, a smooth projective surface with ample anticanonical divisor, is isomorphic to either $X_r^2$ for $0 \leq r \leq 8$ or $\P^1 \times \P^1$.
To study the ring of invariance when it is finitely generated, Hilbert proved his famous syzygy theorem, and introduced the Hilbert function and the Hilbert polynomial. Thus it is quite natural to study syzygies of Cox rings of del Pezzo surfaces.

The obvious first step to study syzygies is to study generators.
Starting with the pioneering work of Batyrev and Popov (\cite{BP}), there has been considerable interest in systematic understanding the generators of ideals of Cox rings of del Pezzo surfaces (see e.g., \cite{P}, \cite{D}, \cite{STV}, \cite{LV}, \cite{SS}, \cite{TVAV}, \cite{SX}). The main problem was known as the Batyrev-Popov conjecture posed in \cite{BP}. To state the conjecture precisely, we introduce some notations. Let $S_r:=X_r^2$ be a del Pezzo surface of degree $(-K_{S_r})^2=9-r$. It is well known that if $r \leq 3$, then $S_r$ is a toric variety so that $\Cox (S_r)$ is a polynomial ring. Thus from now on we assume that $4 \leq r \leq  8$.
It was shown in \cite{BP} that $\Cox(S_r)$ is generated by distinguished global sections, which are corresponding sections of $(-1)$-curves and basis of $H^0(S_8, \mathcal{O}_{S_8}(-K_{S_8}))$. Let $G_r$ be a basis of the vector space $\bigoplus_{E: \text{$(-1)$-curve}} H^0(S_r, E) $ for $r \leq 7$ or
$\left(\bigoplus_{E: \text{$(-1)$-curve}} H^0(S_8, E) \right) \oplus H^0(S_8, -K_{S_8})$. Then we can write
$$
\Cox(S_r) = k[G_r]/I_r
$$
where $I_r$ is the ideal of relations between the generators of $\Cox(S_r)$. Cox rings in general are usually considered as Picard group graded rings, but for Cox rings of del Pezzo surfaces, we have a natural $\Z$-grading by taking intersection with the anticanonical divisor class. See Subsection \ref{coxsubsec} for more detials on Cox rings of del Pezzo surfaces.
The Batyrev-Popov conjecture is the following.

\begin{conjecture}[Batyrev-Popov]
$I_r$ is generated by quadrics
\end{conjecture}

This conjecture is finally confirmed by \cite{SX} and \cite{TVAV}.

A possible next step to study syzygies is  to compute the \emph{Green-Lazarsfeld index}, which was introduced in \cite{BCR} after Green and Lazarsfeld's works (\cite{GL1} and \cite{GL2}). It is defined to be the largest integer $p$ such that $b_{i,j}=0$ for all $i \leq p$ and $j \geq i+2$, where $b_{i,j}$ is the graded Betti number. We denote it by $\index(\Cox(S_r))$.
The Batyrev-Popov conjecture can be restated as $\index(\Cox(S_r)) \geq 1$ for $r \geq 4$.
We prove the following.

\begin{theorem}[Theorem \ref{index}]\label{indexthm}
Let $S_r$ be a del Pezzo surface of degree $9-r$. Then
$\index(\Cox(S_4))=2$ and $\index(\Cox(S_r))=1$ for $5 \leq r \leq 8$.
\end{theorem}

We first note that $\index(S_r)) \geq 1$ is equivalent to the Batyrev-Popov conjecture. We further argue to obtain the upper bound for $\index(S_r)$.
See Subsections \ref{bpsubsec} and \ref{glsubsec} for the complete proof.
Our proof is elementary and heavily depends on the computation of some syzygetic invariants of $\Cox(S_r)$.

The main purpose of the present paper is to calculate important invariants of Cox rings of del Pezzo surfaces such as the \emph{Castelnuovo-Mumford regularity} $\reg(\Cox(S_r))$, the \emph{projective dimension} $\pd(\Cox(S_r))$, the \emph{Hilbert function} $H_{\Cox(S_r)}(t)$, and the \emph{Hilbert polynomial} $P_{\Cox(S_r)}(t)$.

\begin{theorem}[Corollaries \ref{hilft=pol} and \ref{regpd} and Theorem \ref{hilftthm}]\label{mainthm}
Let $S_r$ be a del Pezzo surface of degree $9-r$ where $4 \leq r \leq 8$. Then we have the following.\\
$(1)$ $\reg(\Cox(S_r))=2(r-3)$.\\
$(2)$ $\pd(\Cox(S_r))=|G_r|-r-3$.\\
$(3)$ The Hilbert function $H_{\Cox(S_r)}(t)$ and the Hilbert polynomial $P_{\Cox(S_r)}(t)$ coincide, and they are given as follows.
{\small
\[
\begin{array}{c|c}
r & H_{\Cox(S_r)}(t)=P_{\Cox(S_r)}(t) \\
\hline
4 &\frac{1}{6!}(5t^6 + 75 t^5 + 455t^4 + 1425t^3 + 2420t^2 + 2100 t + 720)\\
\hline
5 & \frac{1}{7!}(34t^7 + 476t^6+2884t^5+9800t^4 +20146t^3+25004t^2+17256t+5040)\\
\hline
6 &\frac{1}{8!}(372t^8 + 4464t^7 + 25200 t^6 + 86184 t^5 \\
&+ 193788 t^4 + 291816 t^3 + 284640 t^2 + 161856 t + 40320)\\
\hline
7 &\frac{1}{9!}(9504 t^9 + 85536 t^8 + 412992 t^7 + 1294272 t^6 + 2860704 t^5  + 4554144 t^4
\\
&+ 5125248 t^3 + 3863808 t^2 + 1752192 t + 362880)\\
\hline
8 &\frac{1}{10!}(1779840 t^{10} + 8899200 t^9 + 32140800 t^8 + 75168000 t^7 + 137531520 t^6\\
   & + 186883200 t^5+ 191635200 t^4 + 141696000 t^3  + 74183040 t^2 +  24624000 t + 3628800)\\
\end{array}
\]
}
\end{theorem}

For the proof of Theorem \ref{mainthm}, see Subsection \ref{invsubsec} and Section \ref{hilsec} whose main ingredient is Popov's geometric result that $\Proj \Cox(S_r) \subset \P^{|G_r|-1}$ is arithmetically Gorenstein (\cite{P}; Proposition \ref{popov}).

Finally, as an application of Theorem \ref{mainthm}, we determine the Betti diagrams of $\Cox(S_r)$ for $r=4,5$.

\begin{theorem}[Theorems \ref{bettiS_4} and \ref{bettiS_5}]\label{bettithm}
We have the following.\\
$(1)$ The Betti diagram of $\Cox(S_4)$ is given as follows.
\noindent\[
\begin{array}{cccc}
1 & - & - & - \\
- & 5 & 5 & -\\
- & - & - & 1
\end{array}
\]
$(2)$ The Betti diagram of $\Cox(S_5)$ is given as follows.
\noindent\[
\begin{array}{ccccccccc}
1 & - & - & - & - & - & - & - & -\\
- & 20 & 48 & 3 & -   & -   & -   & -   &-\\
- & - & 10 & 176 & 280   & 176   & 10   & -   &-\\
- & - & -  & -     & -       & 3   & 48   & 20   &-\\
- & - & - & - & -   & -   & -   & -   &1
\end{array}
\]
\end{theorem}

We point out that the symmetry of the Betti diagrams is not an accident (see Corollary \ref{greendual} for the precise form of the duality).
Furthermore we also compute the $\Pic(S_r)$-graded minimal free resolutions of  $\Cox(S_r)$ for $r=4,5$. See Section \ref{syzdeg4subsec} for more details.

\medskip

The remianing part of this paper is organized as follows. We start in Section \ref{prelimsec} by reviewing known facts on del Pezzo surfaces, Cox rings, and syzygies of graded modules. In Section \ref{gorsec}, we then explain basic properties of Cox rings of del Pezzo surfaces. In particular, we show that Cox rings of del Pezzo surfaces are Gorenstein (Proposition \ref{popov}), and present some consequences in Subsection \ref{invsubsec} (Corollaries \ref{hilft=pol} and \ref{regpd}). Section \ref{hilsec} is devoted to the computation of the Hilbert functions of Cox rings of del Pezzo surfaces (Theorem \ref{hilftthm}). In Section \ref{appsec}, we compute the Green-Lazarsfeld indices of Cox rings of del Pezzo surfaces (Theorem \ref{index}). Finally, in Section \ref{syzdeg4subsec}, we determine the multigraded minimal free resolutions of $\Cox(S_4)$ and $\Cox(S_5)$ and consequently the Betti diagrams of them (Theorems \ref{bettiS_4} and \ref{bettiS_5}).

Throughout the paper, we work over an algebraically closed field $k$ of characteristic zero.

\section{Preliminaries}\label{prelimsec}

In this section, we collect some basic notions and facts which will be used throughout the paper.

\subsection{Del Pezzo surfaces}\label{delsubsec}
We briefly review basic properties of del Pezzo surfaces.

\begin{definition}
A \emph{del Pezzo surface} $S$ is a smooth projective surface over $k$ such that the anticanonical divisor $-K_S$ is ample. The \emph{degree} of a del Pezzo surface $S$ is the integer $(-K_S)^2$.
\end{definition}

A del Pezzo surface is isomorphic to either a blow-up of $\P^2$ at $r$ points in general position for $0 \leq r \leq 8$ or $\P^1 \times \P^1$. Here we fix some notations. Let $S_r$ be a del Pezzo surface obtained by a blow-up $\pi \colon S_r \to \P^2$ at $r$ points in general position for $0 \leq r \leq 8$ with exceptional divisors $E_1, \ldots, E_r$, and let $L:=\pi^*\mathcal{O}_{\P^2}(1)$.
Then $\Pic(S_r)=\Z[L]\oplus\Z[E_1] \oplus \cdots \oplus \Z[E_r]$, and $-K_{S_r}=3L-\sum_{i=1}^r E_i$.

\begin{definition}
An irreducible curve $C$ on a smooth projective surface $S$ is called a \emph{$(-1)$-curve} if $C^2=K_S.C=-1$.
\end{definition}

Recall that all $(-1)$-curves on a del Pezzo surface $S_r$ are given as follows (up to permutation).
\noindent\[
\begin{array}{l}
E_1  ~~(r \geq 1)\\
L-E_1-E_2 ~~(r \geq 2)\\
2L-E_1-E_2-E_3-E_4-E_5 ~~(r \geq 5)\\
3L-2E_1-E_2-E_3-E_4-E_5-E_6-E_7 ~~(r \geq 7)\\
4L-2E_1-2E_2-2E_3-E_4-E_5-E_6-E_7-E_8 ~~(r \geq 8)\\
5L-2E_1-2E_2-2E_3-2E_4-2E_5-2E_6-E_7-E_8 ~~(r \geq 8)\\
6L-3E_1-2E_2-2E_3-2E_4-2E_5-2E_6-2E_7-2E_8 ~~(r \geq 8)\\
\end{array}
\]
The number of $(-1)$-curves on $S_r$ is given as follows.
\noindent\[
\begin{array}{c|c|c|c|c|c|c|c|c|c}
r & 0& 1 & 2& 3& 4 & 5 & 6 & 7 & 8\\
\hline
\text{number of $(-1)$-curves} &0 & 1 & 3 & 6 & 10 & 16 & 27 & 56 & 240
\end{array}
\]

 From now on we assume that $3 \leq r \leq 8$ for convenience.

\begin{definition}
We call a nef divisor $Q$ on $S_r$ a \emph{conic} if $Q^2=0, -K_{S_r}.Q=2$ and a \emph{twisted cubic} if $C^2=1, -K_{S_r}.C=3$.
\end{definition}

Note that a plane conic (resp. twisted cubic) in $\P^3$ lying on a cubic surface $S_6 \subset \P^3$ is actually a conic (resp. twisted cubic) divisor on $S_6$.

For instance, the divisor $L-E_1$ is a conic on $S_r$, and the divisor $L$ is a twisted cubic on $S_r$.
The complete linear system of a conic $Q$ on $S_r$ induces a fibration $f \colon S_r \to \P^1$ such that $Q=f^*\mathcal{O}_{\P^1}(1)$, and that of a twisted cubic $C$ on $S_r$ induces a morphism $\pi \colon S_r \to \P^2$ such that $C=\pi^*\mathcal{O}_{\P^2}(1)$.
The numbers of conics and twisted cubics on $S_r$ are given as follows.
\noindent\[
\begin{array}{c|c|c|c|c|c|c}
r & 3& 4 & 5 & 6 & 7 & 8\\
\hline
\text{number of conics} & 3 & 5 &  10 & 27 & 126 & 2160 \\
\hline
\text{number of twisted cubics} & 2 & 5 & 16 & 72 & 576 & 17520
\end{array}
\]

\begin{proposition}[{\cite[Proposition 3.1 and Lemma 3.4]{TVAV}}]
The following hold.\\
$(1)$ The effective cone $\Eff(S_r$) is generated by $(-1)$-curve classes.\\
$(2)$ The nef cone $\Nef(S_r)$ is generated by classes of conics, twisted cubics, and $-K_{S_6},-K_{S_7},-K_{S_8}$.
\end{proposition}

We can define a root system $R_r \subset \Pic(S_r)$ as
$$
R_r := \{[D] \in \Pic(S_r) \mid D^2=-2, -K_{S_r}.D=0 \}.
$$
Then we obtain the following table.
\noindent\[
\begin{array}{c|c|c|c|c|c}
R_3 & R_4 & R_5 & R_6 & R_7 & R_8\\
\hline
A_1 \times A_2 & A_4 & D_5 & E_6 & E_7 & E_8
\end{array}
\]
Let $W_r$ be the Weyl group of the root system $R_r$. Then we can naturally extend the action of $W_r$ on $[-K_{S_r}]^{\perp}$ to $\Pic(S_r)$ by fixing $[-K_{S_r}]$.

\subsection{Cox rings}
We recall the definition of Cox ring.

\begin{definition}
Let $X$ be a smooth projective variety with finitely generated Picard group. We choose generators $L_1, \ldots, L_m$ of $\Pic(X)$. The \emph{Cox ring} of $X$ with respect to $L_1, \ldots, L_m$ is a $\Pic(X)$-graded ring defined by
$$
\Cox(X) := \bigoplus_{(a_1, \ldots, a_m) \in \Z^m} H^0(X, a_1L_1+ \cdots + a_mL_m).
$$
For $D \in \Pic(X)$, the \emph{degree $D$ part} of $\Cox(X)$ is denoted by $\Cox(X)_D = H^0(X, D)$.
\end{definition}

The ring structure on $\Cox(X)$ certainly depends on the choice of generators of $\Pic(X)$, but the finite generation of $\Cox(X)$ is independent of the choice of generators of $\Pic(X)$ (see \cite[Remark in p.341]{HK}).

\begin{remark}
We will see in Subsection \ref{coxsubsec} that there is a natural choice of generators of $\Pic(S_r)$ so that we can fix the ring structure on the Cox ring of a del Pezzo surface $S_r$.
\end{remark}

Note that $\Cox(X)$ is finitely generated if and only if $X$ is a Mori dream space (\cite[Proposition 2.9]{HK}). Typical examples of Mori dream spaces are toric varieties, regular varieties with Picard number one, and varieties of Fano type.
For further details on Cox rings and Mori dream spaces, we refer to \cite{HK}.

\subsection{Syzygies}\label{syzsubsec}
We recall basic notions in syzygies. For more details, see \cite{E}.
Let $R:=k[x_0, \ldots, x_N]$ be a polynomial ring over $k$ and let $M$ be a finitely generated $\Z$-graded $R$-module.
By Hilbert's syzygy theorem, there is a unique \emph{minimal free resolution} of $M$ with a finite length
$$
F_0 \longleftarrow F_1 \longleftarrow \cdots \longleftarrow F_{\pd(M)} \longleftarrow 0
$$
with $F_i = \bigoplus_{j} S(-j)^{b_{i,j}(M)}$.

\begin{definition}
We call $b_{i,j}(M)$ the \emph{graded Betti number} of $M$.
The graded Betti numbers forms the \emph{Betti diagram} (or \emph{Betti table}) of $M$ as follows.
\begin{displaymath}
\begin{array}{llcl}
b_{0,0} (M) & b_{1,1}(M) & \cdots & b_{\pd(M), \pd(M)} (M) \\
b_{0,1} (M) & b_{1,2} (M) & \cdots & b_{\pd(M), \pd(M)+1} (M) \\
~~~~~\vdots & ~~~~~\vdots & \ddots & ~~~~~~~~~~\vdots \\
b_{0,\reg(M)} (M) & b_{1,\reg(M)+1} (M) & \cdots & b_{\pd(M)+1, \pd(M)+\reg(M)+1} (M)
\end{array}
\end{displaymath}
Here $\pd(M)$ is the \emph{projective dimension} of $M$ and $\reg(M)$ is the \emph{Castelnuovo-Mumford regularity} of $M$.
\end{definition}

\begin{remark}
When $R$ and $M$ are multigraded, one can easily define the \emph{multigraded minimal free resolution} and \emph{multigraded Betti numbers}. For more details on multigraded syzygies, we refer to \cite[Chapter 8]{MS}.
\end{remark}

\begin{definition}
We say that $M$ satisfies the \emph{$N_p$ property} if $b_{i,j}=0$ for all $i \leq p$ and $j \geq i+2$. The \emph{Green-Lazarsfeld index} of $M$ is the largest $p$ such that $M$ satisfies the $N_p$ property.
 We denote it by $\index(M)$.
\end{definition}

\begin{remark}
Let $X \subset \P^N$ be a projective variety and  $R_X:=\bigoplus_{m \in Z} H^0(X, \mathcal{O}_X(m))$ be the section ring. Then $R_X$ is a finitely generated $\Z$-graded $S$-module where $S$ is the homogeneous coordinate ring of $\P^N$. It is easy to see that $R_X$ satisfies $N_0$ property if and only if $X \subset \P^N$ is projectively normal, and $R_X$ satisfies $N_1$ property if and only if $X \subset \P^N$ is projectively normal and its defining ideal is generated by quadrics (see \cite[Subsection 3a]{GL1}). For further details, we refer \cite{GL1}, \cite{GL2} and \cite{EGHP}.
\end{remark}

\begin{definition}
The \emph{Hilbert function} $H_M \colon \Z_{\geq 0} \to \Z_{\geq 0}$ is defined as $H_M(t):=\dim_{k} M_t$. Hilbert proved that there is a polynomial $P_M(t)$ of $t$ such that $P_M(t)=H_M(t)$ for $t \gg 0$.
We call $P_M(t)$ the \emph{Hilbert polynomial}.
The \emph{Hilbert series} is $HS_M(t):=\sum H_M(k)t^k$.
\end{definition}

The Hilbert function can be computed from the Betti diagram as follows.
Let $B_j(M):=\sum_{i \geq 0} (-1)^i b_{i,j}(M)$. It is easy to see that
$$
H_M(j)=B_j(M) + \sum_{k < j} B_k(M) {{N+j-k} \choose {N}}.
$$
Conversely, we also have
$$
B_j(M) = H_M(j) - \sum_{k<j} B_k(M){{N+j-k} \choose {N}}.
$$
Hilbert also proved that the Hilbert series $HS_M(t)$ is a rational function $\frac{K_M(t)}{(1-t)^{N+1}}$, and the \emph{K-polynomial} $K_M(t)$ is given by $K_M(t)=\sum B_j(M) t^j$.

\section{Gorensteinness of Cox rings of del Pezzo surfaces}\label{gorsec}

In this section, we explain basic properties of Cox rings of del Pezzo surfaces. More precisely, we present the generators of Cox rings of del Pezzo surfaces (Proposition \ref{gencoxdp}) and show that they are Gorenstein rings (Proposition \ref{popov}). As consequences we calculate some syzygetic invariants of them (Corollary \ref{regpd}).

\subsection{Cox rings of del Pezzo surfaces}\label{coxsubsec}
We briefly review basics of Cox rings of del Pezzo surfaces.
Here we define Cox rings of del Pezzo surfaces and fix the notations which will be used in the remaining of this paper.

Let $S_r$ be a del Pezzo surface of degree $9-r$.
Recall that $L, E_1, \ldots, E_r$ generate $\Pic(S_r)$. Using these generators as in \cite{BP} and \cite{TVAV},  we give a $\Pic(S_r)$-graded ring structure on $\Cox(S_r)$ as
$$
\Cox(S_r):= \bigoplus_{(a_0, \ldots, a_r )\in \Z^{r+1}} H^0(S_r, a_0L + a_1 E_1 + \cdots +a_r E_r).
$$
There is another natural grading on $\Cox(S_r)$. For a homogeneous element $s \in \Cox(S_r)$ with respect to the $\Pic(S_r)$-grading, there is a line bundle $D$ such that $s \in H^0(S_r, D)$. Then we define $\deg(s):=-K_{S_r}.D$. This grading gives rise to the $\Z$-grading on $\Cox(S_r)$. We call this grading as \emph{the anticanonical grading} on $\Cox(S_r)$.

It is well known that for $0 \leq r \leq 3$, the del Pezzo surface $S_r$ is a toric variety so that $\Cox(S_r)$ is a polynomial ring. Throughout this section, we assume that $4 \leq r \leq 8$.

\begin{definition}
A \emph{distinguished global section} is a global section $s \in H^0(S_r, D)$ for some effective divisor $D$ such that $D$ is a $(-1)$-curve or $-K_{S_8}$.
We denote  $G_r$ by a basis of the vector space $\bigoplus_{E: \text{$(-1)$-curve}} H^0(S_r, E) $ for $r \leq 7$ or
$\left(\bigoplus_{E: \text{$(-1)$-curve}} H^0(S_8, E) \right) \oplus H^0(S_8, -K_{S_8})$.
\end{definition}

The number of elements of $G_r$ is given as follows.
\noindent\[
\begin{array}{c|c|c|c|c|c}
r & 4 & 5 & 6 & 7 & 8\\
\hline
|G_r| & 10 & 16 & 27 & 56 & 242
\end{array}
\]

We write $k[G_r]$ for the $\Pic(S_r)$-graded polynomial ring whose variables are indexed by the elements of $G_r$.

\begin{proposition}[{\cite[Theorem 3.2]{BP}}]\label{gencoxdp}
$\Cox(S_r)$ is generated by distinguished global sections. In particular, there is a canonical surjection $g_r \colon k[G_r] \to \Cox(S_r)$.
\end{proposition}

Let $I_r$ be the kernel of $g_r \colon k[G_r] \to \Cox(S_r)$. Then $I_r$ and $\Cox(S_r)$ are finitely generated $\Pic(S_r)$-graded $k[G_r]$- modules.

By giving $\deg (x):=1$ to any $x \in G_r$, we can regard $k[G_r]$ as a $\Z$-graded polynomial ring. Then  $I_r$ and $\Cox(S_r)$ are also finitely generated $\Z$-graded $k[G_r]$- modules. Furthermore this $\Z$-grading on $\Cox(S_r)$ coincides with the anticanonical grading.

\subsection{Geometry of Proj of Cox rings}\label{popovsubsec}

Recall that $\Cox(S_r)$ has a natural $\Z$-grading, and hence we can define a projective variety
$$
X_r:=\Proj \Cox(S_r) \subset \P^{N_r}
$$
where $N_r=|G_r|-1$.
The aim of this subsection is to study some geometric properties of $X_r \subset \P^{N_r}$.

\begin{proposition}\label{X_rprop}
We have the following.\\
$(1)$ $n_r:=\dim X_r = r+2$.\\
$(2)$ $X_r \subset \P^{N_r}$ is projectively normal.\\
$(3)$ $\Pic(X_r) = \Z [\mathcal{O}_{X_r}(1)]$.\\
$(4)$ $\Cox(S_r) \simeq \bigoplus_{m \in \Z} H^0(X_r, \mathcal{O}_{X_r}(m))$ as a $\Z$-graded ring.
\end{proposition}

\begin{proof}
(1): Note that $\dim X_r = \dim \Cox(S_r)-1 = \rank \Pic(S_r) + \dim S_r -1 = r+2$.\\
(2): It is a consequence of Proposition \ref{gencoxdp}.\\
(3): Since $\Cox(S_r)$ is a unique factorization domain, the assertion follows.\\
(4): Note that $\Cox(S_r)$ is the homogeneous coordinate ring of $X_r \subset \P^{N_r}$. Since  $X_r \subset \P^{N_r}$ is projectively normal, $\Cox(S_r)$ coincides with the section ring of $X_r \subset \P^{N_r}$, and thus we get the assertion.
\end{proof}

Note that the homogenous coordinate ring of $\P^{N_r}$ is nothing but $k[G_r]$. In particular $\Cox(S_r)$ satisfies $N_0$ property.

\begin{remark}
By computation of Hilbert functions of $\Cox(S_r)$ in Section \ref{hilsec}, we can also determine the degree $d_r$ of $X_r \subset \P^{N_r}$ as follows.
\noindent\[
\begin{array}{c|c|c|c|c}
 d_4 & d_5 & d_6 & d_7 & d_8 \\
\hline
 5 & 34 & 372 & 9504 & 1779840 \\
\end{array}
\]
\end{remark}

The following is the main property of $X_r \subset \P^{N_r}$.

\begin{proposition}[Popov]\label{popov}
$X_r \subset \P^{N_r}$ is arithmetically Gorenstein and $-K_{X_r}=\mathcal{O}_{X_r}(9-r)$. In particular $X_r \subset \P^{N_r}$ is arithmetically Cohen-Macaulay, i.e., $\Cox(S_r)$ is Cohen-Macaulay.
\end{proposition}

\begin{proof}
It was originally proved in \cite{P}.
Here we give a simple proof. By \cite[Theorem 1.2]{B} and \cite[Theorem 1.1]{GOST}, $\Cox(S_r)$ has at worst log terminal singularities so that it is a Cohen-Macaulay ring. Now by applying \cite[Corollary 1.5]{HaKu}, we obtain the conclusion.
\end{proof}

\begin{remark}
In this paper, the characteristic zero assumption is used only to prove Proposition \ref{popov}. 
In our proof, we apply results from \cite{B} and \cite{GOST}, which only hold in characteristic zero.
The original proof in \cite{P} also requires such an assumption.
\end{remark}

\begin{remark}
It is well known that $X_4=Gr(2,5) \subset \P^9$ (\cite[Proposition 4.1]{BP}).
Note that $X_5$ has 16 isolated singular points, and consequently, one can show that $X_r$ has singular locus of codimension $7$ for $5 \leq r \leq 8$ (\cite[Corollary 4.5]{BP}). However singularities on $X_r$ are mild. Since $X_r$ is Gorenstein and has at worst log terminal singularities, it has at worst canonical singularities. It would be an interesting problem to study further geometric properties of $X_r$.
\end{remark}

\subsection{Consequences}\label{invsubsec}
We now present some consequences of Proposition \ref{popov}.
First we have the following vanishing statement.

\begin{corollary}\label{vanishing}
$(1)$ $h^i(X_r, \mathcal{O}_{X_r}(k))=0$ for $0 < i < n_r$ and $k \in \Z$.\\
$(2)$ $h^{n_r}(X_r, \mathcal{O}_{X_r}(k))=0$ for $k \geq r-8$.
\end{corollary}

\begin{proof}
Since $X_r \subset \P^{N_r}$ is arithmetically Cohen-Macaulay by Proposition \ref{popov}, we get (1). For (2), we apply the Serre duality and Proposition \ref{popov} so that we have
$$
h^{n_r}(X_r, \mathcal{O}_{X_r}(k))=h^0(X_r, \mathcal{O}_{X_r}(r-9-k)).
$$
The assertion follows.
\end{proof}

Let $H_r(t):=H_{\Cox(S_r)}(t)$ be the Hilbert function, and let $P_r(t):=P_{\Cox(S_r)}(t)$ be the Hilbert polynomial.

\begin{corollary}\label{hilft=pol}
The Hilbert function and the Hilbert polynomial of $\Cox(S_r)$ coincide, i.e., $H_r(t)=P_r(t)$.
\end{corollary}

\begin{proof}
Since $\Cox(S_r) \simeq \bigoplus_{m \in \Z} H^0(X_r, \mathcal{O}_{X_r}(m))$ by Proposition \ref{X_rprop}, we have $H_r(t)=h^0(X_r, \mathcal{O}_{X_r}(t))$ for $t \geq 0$. By Corollary \ref{vanishing}, $h^0(X_r, \mathcal{O}_{X_r}(t)) = \chi (\mathcal{O}_{X_r}(t))$ for $t \geq 0$. It is well known that $\chi (\mathcal{O}_{X_r}(t))=P_r(t)$ for $t \geq 0$. Hence $H_r(t)=P_r(t)$.
\end{proof}

We can also compute the Castelnuovo-Mumford regularity and the projective dimension of $\Cox(S_r)$.

\begin{corollary}\label{regpd}
$\reg(\Cox(S_r))=2(r-3)$ and $\pd(\Cox(S_r))=\codim(X_r, \P^{N_r})=N_r-r-2$.
\end{corollary}

\begin{proof}
Since $\Cox(S_r) \simeq \bigoplus_{m \in \Z} H^0(X_r, \mathcal{O}_{X_r}(m))$ by Proposition \ref{X_rprop}, we obtain $\reg(\Cox(S_r))=\reg(\mathcal{O}_{X_r})$. Using Corollary \ref{vanishing}, we obtain $\reg(\Cox(S_r))=2(r-3)$.
For the projective dimension, first note that $X_r \subset \P^{N_r}$ is arithmetically Cohen Macaulay by Proposition \ref{popov}. Then we obtain $\pd(\Cox(S_r))=\codim(X_r, \P^{N_r})=N_r-r-2$.
\end{proof}

We also have the following duality of graded Betti numbers.

\begin{corollary}\label{greendual}
$b_{i,j}(\Cox(S_r))=b_{N_r-r-i-2, r+N_r-j-8}(\Cox(S_r))$.
\end{corollary}

\begin{proof}
The assertion follows from Green's duality theorem (\cite[Theorem 2.c.6]{G}).
\end{proof}

We will see the multigraded version duality in Section \ref{syzdeg4subsec} (see Lemma \ref{dual}).

\section{Hilbert functions of Cox rings of del Pezzo surfaces}\label{hilsec}

This section is the main part of the present paper and is devoted to computation of Hilbert functions of Cox rings of del Pezzo surfaces. Let $S_r$ be a del Pezzo surface of degree $9-r$. We assume that $4 \leq r \leq 8$.

\begin{theorem}\label{hilftthm}
The Hilbert functions $H_{\Cox(S_r)}(t)$ of Cox rings of del Pezzo surfaces are  given as follows.
{\small
\[
\begin{array}{c|c}
r & H_{\Cox(S_r)}(t)=P_{\Cox(S_r)}(t) \\
\hline
4 &\frac{1}{6!}(5t^6 + 75 t^5 + 455t^4 + 1425t^3 + 2420t^2 + 2100 t + 720)\\
\hline
5 & \frac{1}{7!}(34t^7 + 476t^6+2884t^5+9800t^4 +20146t^3+25004t^2+17256t+5040)\\
\hline
6 &\frac{1}{8!}(372t^8 + 4464t^7 + 25200 t^6 + 86184 t^5 \\
&+ 193788 t^4 + 291816 t^3 + 284640 t^2 + 161856 t + 40320)\\
\hline
7 &\frac{1}{9!}(9504 t^9 + 85536 t^8 + 412992 t^7 + 1294272 t^6 + 2860704 t^5  + 4554144 t^4
\\
&+ 5125248 t^3 + 3863808 t^2 + 1752192 t + 362880)\\
\hline
8 &\frac{1}{10!}(1779840 t^{10} + 8899200 t^9 + 32140800 t^8 + 75168000 t^7 + 137531520 t^6\\
   & + 186883200 t^5+ 191635200 t^4 + 141696000 t^3  + 74183040 t^2 +  24624000 t + 3628800)\\
\end{array}
\]
}
\end{theorem}

The organization of arguments for the proof of this theorem is as follows. First, we explain our strategy, and then, we present actual calculations.

\subsection{Strategy}
Since we know that the Hilbert function and the Hilbert polynomial of $\Cox(S_r)$ coincide, we naively think that it is enough to calculate $h^0(X_r, \mathcal{O}_{X_r}(k))$ for $\dim X_r + 1=r+3$ many different values $k$. To reduce the calculations, we further use geometric properties of $X_r \subset \P^{N_r}$ (see Subsection \ref{popovsubsec}).

Since $X_r \subset \P^{N_r}$ is arithmatically Cohen-Macaulay (Proposition \ref{popov}), so is every general linear section $X$. Thus we obtain $h^1(X, \mathcal{O}_X(k))=0$ for all $k \in \Z$ when $\dim X \geq 2$. Then for two successive linear sections $Y \subset X$ of $X_r \subset \P^{N_r}$, we have
$$
h^0(X, \mathcal{O}_X(k))=h^0(X, \mathcal{O}_X(k-1))+h^0(Y, \mathcal{O}_Y(k)).
$$
Recall that $X_r$ has at worst canonical singularities so that $X_r$ is, in particular, normal. Thus a general curve section $C_r \subset \P^{N_r-\dim X_r +1}$ is smooth, and $K_{C_r}=\mathcal{O}_{C_r}(2(r-4))$. By {Riemann-Roch formula}, it suffices to compute {degree $d_r$, genus $g_r$, and $h^0(C_r, \mathcal{O}_{C_r}(r-4))$}. We have $g_r=(r-4)d_r+1$, and
$$
h^0(C_r, \mathcal{O}_{C_r}(r-3))-h^0(C_r, \mathcal{O}_{C_r}(r-5))=(r-3)d_r-g_r+1=d_r.
$$
Hence we only have to calculate $h^0(C_r, \mathcal{O}_{C_r}(r-5))$ and $h^0(C_r, \mathcal{O}_{C_r}(r-3))$.
For this purpose, we actually compute $h^0(X_r, \mathcal{O}_{X_r}(1)), \ldots, h^0(X_r, \mathcal{O}_{X_r}(r-3))$. Note that
$$
h^0(X_r, \mathcal{O}_{X_r}(k)) = \sum_{-K_{S_r}.D=k} h^0(S_r, \mathcal{O}_{S_r}(D)).
$$
Let $\deg D := -K_{S_r}.D$ for a divisor $D$ on $S_r$.
Thus our actual computation is carried out by the following order.
\begin{enumerate}
\item[Step 1.] Make up a list of all effective divisors $D$ up to linear equivalence on $S_r$ with $1\leq \deg D \leq r-3$.
\item[Step 2.] Compute $h^0(S_r, \mathcal{O}_{S_r}(D))$ for all effective divisors $D$ in Step 1.
\end{enumerate}

For Step 2, we use the next lemma.

\begin{lemma}\label{computeh0lem}
We have the following.\\
$(1)$ If $D$ is not nef, then there is a $(-1)$-curve $E$ with $D.E<0$. In this case, $h^0(S_r, \mathcal{O}_{S_r}(D))=h^0(S_r, \mathcal{O}_{S_r}(D-E))$ and $\deg (D-E)=\deg D-1$.\\
$(2)$ If $D$ is nef, then $h^0(S_r, \mathcal{O}_{S_r}(D))=\frac{D^2+\deg D+2}{2}$.

\end{lemma}

\begin{proof}
Since $\Eff(S_r)$ is generated by $(-1)$-curves, we can easily check the nefness of a divisor and the first assertion of (1) follows. If $D.E<0$ for some $(-1)$-curve $E$, then $E$ must be a fixed component of $|D|$, so we obtain the remaining part of (1).
For (2), we note that $-K_{S_r}+D$ is ample when $D$ is nef. By Kodaira vanishing theorem, we get $h^i(S_r, \mathcal{O}_{S_r}(D))=0$ for $i >0$. Now (2) follows from Riemann-Roch formula.
\end{proof}

\subsection{Computation}

The computation of the values  $h^0(X_r,\mathcal{O}_{X_r}(t))$ yields by the following manners.

In the cases of $r=4,5$ and 6, consider $(-1)$-curves of types $E_1$ , $L-E_1-E_2$ and  $2L-E_1-E_2-E_3-E_4-E_5$ up to permutation and denote these types by $l_0,l_1$ and $l_2$, respectively, depending on the coefficients of $L$.
Fixing an anticanonical degree   and a coefficient of $L$ ,  we  investigate all possible summation of $(-1)$-curves. In other words, the divisor  $aL+b_1E_1+\cdots +b_rE_r$ with anticanonical degree $t$ is represented by $l_{i_1}+\cdots + l_{i_t}$, where $i_1+\cdots+ i_t=a$ . Thus  by considering  all possible combinations of $l_i$'s, we can collect all informations of coefficients  $(a;b_1,\ldots b_r)$.

Note that $h^0(X_r, \mathcal{O}_{X_r}(1))$ is nothing but the number of $(-1)$-curves on $S_r$.

\subsubsection{$r=4$} We have $h^0(X_4, \mathcal{O}_{X_4}(1))=10$, and consequently, we get $d_4=5$ and $g_4=1$.

\subsubsection{$r=5$}
We have $h^0(X_5, \mathcal{O}_{X_5}(1))=16$ and 
$$
h^0(X_5, \mathcal{O}_{X_5}(2))=5+10+30+10+10+30+10+10+1=116
$$ 
by the following table. Thus we get $d_5=34$ and $g_5=35$.

We read the tables in the following way: In the first column, the vector $(a;b_1,\ldots b_r)$ corresponds
to the divisor $D=aL+b_1E_1+\cdots + b_rE_r$. And the second
column gives the number of divisors that is symmetric to the
divisor appeared in the first column.  Namely it presents the number of vectors $(b'_1,\ldots, b'_r)$  such that $\{b'_1,\ldots, b'_r\}=\{b_1,\ldots, b_r\}$.
All those divisors have same values of  $h^0(S_r, \mathcal{O}_{S_r}(D))$. 
The total means the number of divisors times $h^0(D)$. Thus $h^0(X_r, \mathcal{O}_{X_r}(k))$ can be obtained by the summation of the numbers in the column named by total when $-K_{S_r}\cdot D=k$.

We now explain how to obtain the following tables. First, we list up all possible types of divisors.
This is an elementary combinatorics problem.
 Then we compute the numbers of divisors in each type, and compute $h^0(D)$ using Lemma \ref{computeh0lem}. 
For example, when the type of a divisor is $(2;-1,-1,-2,0,0)$, the number of $2L-E_i - E_j-2E_k$ such that $1 \leq i,j,k \leq 5$ and $i,j,k$ are distinct is ${5 \choose 2} \times 3=30$. 
Note that $h^0(2L-E_i-E_j-2E_k) = h^0(2L-E_1-E_2-2E_3)$. Since $(2L-E_1-E_2-2E_3)\cdot(L-E_1-E_3)=-1$, we get $h^0(2L-E_1-E_2-2E_3)=h^0(L-E_2-E_3)=1$ by Lemma \ref{computeh0lem}.

%%%%%%%%%%%%%%%%%%%$-K_{X_5}\cdot D=2$
\begin{longtable}{|l|c|c|c|}
\hline \multicolumn{4}{|c|}
{$-K_{S_5}\cdot D=2$} \\
\hline
\begin{minipage}[m]{.40\linewidth}\setlength{\unitlength}{.25mm}

\begin{center}
\medskip
types $(L ; E_1, E_2, E_3, E_4, E_5)$
\medskip
\end{center}
\end{minipage}  & number of divisors &  $h^0(D)$& total\\

\hline

\begin{minipage}[m]{.40\linewidth}\setlength{\unitlength}{.25mm}
\begin{center}
(0 ; 2, 0, 0, 0, 0)
\end{center}
\end{minipage}
& \begin{minipage}[m]{.10\linewidth}\setlength{\unitlength}{.25mm}
\begin{center}
5
\end{center}
\end{minipage}& \begin{minipage}[m]{.10\linewidth}\setlength{\unitlength}{.25mm}
\begin{center}
1
\end{center}
\end{minipage} & \begin{minipage}[m]{.10\linewidth}\setlength{\unitlength}{.25mm}
\begin{center}
5
\end{center}
\end{minipage}\\

\hline

\begin{minipage}[m]{.40\linewidth}\setlength{\unitlength}{.25mm}
\begin{center}
(0 ; 1, 1, 0, 0, 0)
\end{center}
\end{minipage}
& \begin{minipage}[m]{.10\linewidth}\setlength{\unitlength}{.25mm}
\begin{center}
10
\end{center}
\end{minipage}& \begin{minipage}[m]{.10\linewidth}\setlength{\unitlength}{.25mm}
\begin{center}
1
\end{center}
\end{minipage} & \begin{minipage}[m]{.10\linewidth}\setlength{\unitlength}{.25mm}
\begin{center}
10
\end{center}
\end{minipage}\\

\hline

\begin{minipage}[m]{.40\linewidth}\setlength{\unitlength}{.25mm}
\begin{center}
(1 ; 1, -1, -1, 0, 0)
\end{center}
\end{minipage}
& \begin{minipage}[m]{.10\linewidth}\setlength{\unitlength}{.25mm}
\begin{center}
30
\end{center}
\end{minipage}& \begin{minipage}[m]{.10\linewidth}\setlength{\unitlength}{.25mm}
\begin{center}
1
\end{center}
\end{minipage} & \begin{minipage}[m]{.10\linewidth}\setlength{\unitlength}{.25mm}
\begin{center}
30
\end{center}
\end{minipage}\\

\hline

\begin{minipage}[m]{.40\linewidth}\setlength{\unitlength}{.25mm}
\begin{center}
(1 ; -1, 0, 0, 0, 0)
\end{center}
\end{minipage}
& \begin{minipage}[m]{.10\linewidth}\setlength{\unitlength}{.25mm}
\begin{center}
5
\end{center}
\end{minipage}& \begin{minipage}[m]{.10\linewidth}\setlength{\unitlength}{.25mm}
\begin{center}
2
\end{center}
\end{minipage} & \begin{minipage}[m]{.10\linewidth}\setlength{\unitlength}{.25mm}
\begin{center}
10
\end{center}
\end{minipage}\\

\hline

\begin{minipage}[m]{.40\linewidth}\setlength{\unitlength}{.25mm}
\begin{center}
(2 ; -2, -2, 0, 0, 0)
\end{center}
\end{minipage}
& \begin{minipage}[m]{.10\linewidth}\setlength{\unitlength}{.25mm}
\begin{center}
10
\end{center}
\end{minipage}& \begin{minipage}[m]{.10\linewidth}\setlength{\unitlength}{.25mm}
\begin{center}
1
\end{center}
\end{minipage} & \begin{minipage}[m]{.10\linewidth}\setlength{\unitlength}{.25mm}
\begin{center}
10
\end{center}
\end{minipage}\\

\hline

\begin{minipage}[m]{.40\linewidth}\setlength{\unitlength}{.25mm}
\begin{center}
(2 ; -1, -1, -2, 0, 0)
\end{center}
\end{minipage}
& \begin{minipage}[m]{.10\linewidth}\setlength{\unitlength}{.25mm}
\begin{center}
30
\end{center}
\end{minipage}& \begin{minipage}[m]{.10\linewidth}\setlength{\unitlength}{.25mm}
\begin{center}
1
\end{center}
\end{minipage} & \begin{minipage}[m]{.10\linewidth}\setlength{\unitlength}{.25mm}
\begin{center}
30
\end{center}
\end{minipage}\\

\hline

\begin{minipage}[m]{.40\linewidth}\setlength{\unitlength}{.25mm}
\begin{center}
(2 ; -1, -1, -1, -1,  0)
\end{center}
\end{minipage}
& \begin{minipage}[m]{.10\linewidth}\setlength{\unitlength}{.25mm}
\begin{center}
5
\end{center}
\end{minipage}& \begin{minipage}[m]{.10\linewidth}\setlength{\unitlength}{.25mm}
\begin{center}
2
\end{center}
\end{minipage} & \begin{minipage}[m]{.10\linewidth}\setlength{\unitlength}{.25mm}
\begin{center}
10
\end{center}
\end{minipage}\\

\hline

\begin{minipage}[m]{.40\linewidth}\setlength{\unitlength}{.25mm}
\begin{center}
(3 ; -2, -2, -1, -1,  -1)
\end{center}
\end{minipage}
& \begin{minipage}[m]{.10\linewidth}\setlength{\unitlength}{.25mm}
\begin{center}
10
\end{center}
\end{minipage}& \begin{minipage}[m]{.10\linewidth}\setlength{\unitlength}{.25mm}
\begin{center}
1
\end{center}
\end{minipage} & \begin{minipage}[m]{.10\linewidth}\setlength{\unitlength}{.25mm}
\begin{center}
10
\end{center}
\end{minipage}\\

\hline

\begin{minipage}[m]{.40\linewidth}\setlength{\unitlength}{.25mm}
\begin{center}
(4 ; -2, -2, -2, -2,  -2)
\end{center}
\end{minipage}
& \begin{minipage}[m]{.10\linewidth}\setlength{\unitlength}{.25mm}
\begin{center}
1
\end{center}
\end{minipage}& \begin{minipage}[m]{.10\linewidth}\setlength{\unitlength}{.25mm}
\begin{center}
1
\end{center}
\end{minipage} & \begin{minipage}[m]{.10\linewidth}\setlength{\unitlength}{.25mm}
\begin{center}
1
\end{center}
\end{minipage}\\

\hline

\end{longtable}
%%%%%%%%%%%%%%%%%%%

\subsubsection{$r=6$}
In this cases, the divisor types $l_0$ and $l_2$ play a symmetric
role, since the sum of divisors of types $l_0$ and $l_2$, $(E_i)+
(2L-E_1-\cdots -E_6+E_i)$ is $-K_{S_6}-L$. In other words, when we
consider all possibilities of $l_{i_1}+\cdots + l_{i_t}$ for a
given coefficient of $L$, we have the same if the exchange of the
divisor type $l_0$ and $l_2$ is made.
It means that for a given  an anticanonical degree  $t$,  the divisor types containing a term $aL$ and  the divisor types containing a term $(2t-a)L$ have one to one correspondences.

We have $h^0(X_6, \mathcal{O}_{X_6}(1))=27$. By following tables,
we obtain $h^0(X_6, \mathcal{O}_{X_6}(2))=297$ and $h^0(X_6,
\mathcal{O}_{X_6}(3))=1939$. Thus we get $d_6=372$ and $g_6=745$.

%%%%%%%%%%%%%%%%%%%$-K_{S_6}\cdot D=2$
\begin{longtable}{|l|c|c|c|}
\hline \multicolumn{4}{|c|}
{$-K_{S_6}\cdot D=2$} \\
\hline
\begin{minipage}[m]{.40\linewidth}\setlength{\unitlength}{.25mm}

\begin{center}
\medskip
types $(L ; E_1, E_2, E_3, E_4, E_5, E_6)$
\medskip
\end{center}
\end{minipage}  & number of divisors &  $h^0(D)$& total\\

\hline

\begin{minipage}[m]{.40\linewidth}\setlength{\unitlength}{.25mm}
\begin{center}
(0 ; 2, 0, 0, 0, 0, 0)
\end{center}
\end{minipage}
& \begin{minipage}[m]{.10\linewidth}\setlength{\unitlength}{.25mm}
\begin{center}
6
\end{center}
\end{minipage}& \begin{minipage}[m]{.10\linewidth}\setlength{\unitlength}{.25mm}
\begin{center}
1
\end{center}
\end{minipage} & \begin{minipage}[m]{.10\linewidth}\setlength{\unitlength}{.25mm}
\begin{center}
6
\end{center}
\end{minipage}\\

\hline

\begin{minipage}[m]{.40\linewidth}\setlength{\unitlength}{.25mm}
\begin{center}
(0 ; 1, 1, 0, 0, 0, 0)
\end{center}
\end{minipage}
& \begin{minipage}[m]{.10\linewidth}\setlength{\unitlength}{.25mm}
\begin{center}
15
\end{center}
\end{minipage}& \begin{minipage}[m]{.10\linewidth}\setlength{\unitlength}{.25mm}
\begin{center}
1
\end{center}
\end{minipage} & \begin{minipage}[m]{.10\linewidth}\setlength{\unitlength}{.25mm}
\begin{center}
15
\end{center}
\end{minipage}\\

\hline
\begin{minipage}[m]{.40\linewidth}\setlength{\unitlength}{.25mm}
\begin{center}

\end{center}
\end{minipage}
& \begin{minipage}[m]{.10\linewidth}\setlength{\unitlength}{.25mm}
\begin{center}

\end{center}
\end{minipage}& \begin{minipage}[m]{.10\linewidth}\setlength{\unitlength}{.25mm}
\begin{center}

\end{center}
\end{minipage} & \begin{minipage}[m]{.10\linewidth}\setlength{\unitlength}{.25mm}
\begin{center}
21
\end{center}
\end{minipage}\\

\hline

\begin{minipage}[m]{.40\linewidth}\setlength{\unitlength}{.25mm}
\begin{center}
(1 ; 1, -1, -1, 0, 0, 0)
\end{center}
\end{minipage}
& \begin{minipage}[m]{.10\linewidth}\setlength{\unitlength}{.25mm}
\begin{center}
60
\end{center}
\end{minipage}& \begin{minipage}[m]{.10\linewidth}\setlength{\unitlength}{.25mm}
\begin{center}
1
\end{center}
\end{minipage} & \begin{minipage}[m]{.10\linewidth}\setlength{\unitlength}{.25mm}
\begin{center}
60
\end{center}
\end{minipage}\\

\hline

\begin{minipage}[m]{.40\linewidth}\setlength{\unitlength}{.25mm}
\begin{center}
(1 ; -1, 0, 0, 0, 0, 0)
\end{center}
\end{minipage}
& \begin{minipage}[m]{.10\linewidth}\setlength{\unitlength}{.25mm}
\begin{center}
6
\end{center}
\end{minipage}& \begin{minipage}[m]{.10\linewidth}\setlength{\unitlength}{.25mm}
\begin{center}
2
\end{center}
\end{minipage} & \begin{minipage}[m]{.10\linewidth}\setlength{\unitlength}{.25mm}
\begin{center}
12
\end{center}
\end{minipage}\\

\hline

\begin{minipage}[m]{.40\linewidth}\setlength{\unitlength}{.25mm}
\begin{center}

\end{center}
\end{minipage}
& \begin{minipage}[m]{.10\linewidth}\setlength{\unitlength}{.25mm}
\begin{center}

\end{center}
\end{minipage}& \begin{minipage}[m]{.10\linewidth}\setlength{\unitlength}{.25mm}
\begin{center}

\end{center}
\end{minipage} & \begin{minipage}[m]{.10\linewidth}\setlength{\unitlength}{.25mm}
\begin{center}
72
\end{center}
\end{minipage}\\

\hline

\begin{minipage}[m]{.40\linewidth}\setlength{\unitlength}{.25mm}
\begin{center}
(2 ; -2, -2, 0, 0, 0, 0)
\end{center}
\end{minipage}
& \begin{minipage}[m]{.10\linewidth}\setlength{\unitlength}{.25mm}
\begin{center}
15
\end{center}
\end{minipage}& \begin{minipage}[m]{.10\linewidth}\setlength{\unitlength}{.25mm}
\begin{center}
1
\end{center}
\end{minipage} & \begin{minipage}[m]{.10\linewidth}\setlength{\unitlength}{.25mm}
\begin{center}
15
\end{center}
\end{minipage}\\

\hline

\begin{minipage}[m]{.40\linewidth}\setlength{\unitlength}{.25mm}
\begin{center}
(2 ; -1, -1, -2, 0, 0, 0)
\end{center}
\end{minipage}
& \begin{minipage}[m]{.10\linewidth}\setlength{\unitlength}{.25mm}
\begin{center}
60
\end{center}
\end{minipage}& \begin{minipage}[m]{.10\linewidth}\setlength{\unitlength}{.25mm}
\begin{center}
1
\end{center}
\end{minipage} & \begin{minipage}[m]{.10\linewidth}\setlength{\unitlength}{.25mm}
\begin{center}
60
\end{center}
\end{minipage}\\

\hline

\begin{minipage}[m]{.40\linewidth}\setlength{\unitlength}{.25mm}
\begin{center}
(2 ; -1, -1, -1, -1, 0, 0)
\end{center}
\end{minipage}
& \begin{minipage}[m]{.10\linewidth}\setlength{\unitlength}{.25mm}
\begin{center}
15
\end{center}
\end{minipage}& \begin{minipage}[m]{.10\linewidth}\setlength{\unitlength}{.25mm}
\begin{center}
2
\end{center}
\end{minipage} & \begin{minipage}[m]{.10\linewidth}\setlength{\unitlength}{.25mm}
\begin{center}
30
\end{center}
\end{minipage}\\

\hline

\begin{minipage}[m]{.40\linewidth}\setlength{\unitlength}{.25mm}
\begin{center}
(2 ; 1, -1, -1, -1, -1, -1)
\end{center}
\end{minipage}
& \begin{minipage}[m]{.10\linewidth}\setlength{\unitlength}{.25mm}
\begin{center}
6
\end{center}
\end{minipage}& \begin{minipage}[m]{.10\linewidth}\setlength{\unitlength}{.25mm}
\begin{center}
1
\end{center}
\end{minipage} & \begin{minipage}[m]{.10\linewidth}\setlength{\unitlength}{.25mm}
\begin{center}
6
\end{center}
\end{minipage}\\
\hline

\begin{minipage}[m]{.40\linewidth}\setlength{\unitlength}{.25mm}
\begin{center}

\end{center}
\end{minipage}
& \begin{minipage}[m]{.10\linewidth}\setlength{\unitlength}{.25mm}
\begin{center}

\end{center}
\end{minipage}& \begin{minipage}[m]{.10\linewidth}\setlength{\unitlength}{.25mm}
\begin{center}

\end{center}
\end{minipage} & \begin{minipage}[m]{.10\linewidth}\setlength{\unitlength}{.25mm}
\begin{center}
111
\end{center}
\end{minipage}\\

\hline

\begin{minipage}[m]{.40\linewidth}\setlength{\unitlength}{.25mm}
\begin{center}
(3 ; $\ast$, $\ast$, $\ast$, $\ast$, $\ast$, $\ast$)
\end{center}
\end{minipage}
& \begin{minipage}[m]{.10\linewidth}\setlength{\unitlength}{.25mm}
\begin{center}

\end{center}
\end{minipage}& \begin{minipage}[m]{.10\linewidth}\setlength{\unitlength}{.25mm}
\begin{center}

\end{center}
\end{minipage} & \begin{minipage}[m]{.10\linewidth}\setlength{\unitlength}{.25mm}
\begin{center}
72
\end{center}
\end{minipage}\\
\hline

\begin{minipage}[m]{.40\linewidth}\setlength{\unitlength}{.25mm}
\begin{center}
(4 ; $\ast$, $\ast$, $\ast$, $\ast$, $\ast$, $\ast$)
\end{center}
\end{minipage}
& \begin{minipage}[m]{.10\linewidth}\setlength{\unitlength}{.25mm}
\begin{center}

\end{center}
\end{minipage}& \begin{minipage}[m]{.10\linewidth}\setlength{\unitlength}{.25mm}
\begin{center}

\end{center}
\end{minipage} & \begin{minipage}[m]{.10\linewidth}\setlength{\unitlength}{.25mm}
\begin{center}
21
\end{center}
\end{minipage}\\
\hline

\end{longtable}
%%%%%%%%%%%%%%%%%%%

%%%%%%%%%%%%%%%%%%%$-K_{S_6}\cdot D=3$
\begin{longtable}{|l|c|c|c|}
\hline \multicolumn{4}{|c|}
{$-K_{S_6}\cdot D=3$} \\
\hline

\begin{minipage}[m]{.40\linewidth}\setlength{\unitlength}{.25mm}

\begin{center}
\medskip
types $(L ; E_1, E_2, E_3, E_4, E_5, E_6)$
\medskip
\end{center}
\end{minipage}  & number of divisors &  $h^0(D)$& total\\

\hline

\begin{minipage}[m]{.40\linewidth}\setlength{\unitlength}{.25mm}
\begin{center}
(0 ; 3, 0, 0, 0, 0, 0)
\end{center}
\end{minipage}
& \begin{minipage}[m]{.10\linewidth}\setlength{\unitlength}{.25mm}
\begin{center}
6
\end{center}
\end{minipage}& \begin{minipage}[m]{.10\linewidth}\setlength{\unitlength}{.25mm}
\begin{center}
1
\end{center}
\end{minipage} & \begin{minipage}[m]{.10\linewidth}\setlength{\unitlength}{.25mm}
\begin{center}
6
\end{center}
\end{minipage}\\

\hline

\begin{minipage}[m]{.40\linewidth}\setlength{\unitlength}{.25mm}
\begin{center}
(0 ; 2, 1, 0, 0, 0, 0)
\end{center}
\end{minipage}
& \begin{minipage}[m]{.10\linewidth}\setlength{\unitlength}{.25mm}
\begin{center}
30
\end{center}
\end{minipage}& \begin{minipage}[m]{.10\linewidth}\setlength{\unitlength}{.25mm}
\begin{center}
1
\end{center}
\end{minipage} & \begin{minipage}[m]{.10\linewidth}\setlength{\unitlength}{.25mm}
\begin{center}
30
\end{center}
\end{minipage}\\

\hline

\begin{minipage}[m]{.40\linewidth}\setlength{\unitlength}{.25mm}
\begin{center}
(0 ; 1, 1, 1, 0, 0, 0)
\end{center}
\end{minipage}
& \begin{minipage}[m]{.10\linewidth}\setlength{\unitlength}{.25mm}
\begin{center}
20
\end{center}
\end{minipage}& \begin{minipage}[m]{.10\linewidth}\setlength{\unitlength}{.25mm}
\begin{center}
1
\end{center}
\end{minipage} & \begin{minipage}[m]{.10\linewidth}\setlength{\unitlength}{.25mm}
\begin{center}
20
\end{center}
\end{minipage}\\

\hline

\begin{minipage}[m]{.40\linewidth}\setlength{\unitlength}{.25mm}
\begin{center}

\end{center}
\end{minipage}
& \begin{minipage}[m]{.10\linewidth}\setlength{\unitlength}{.25mm}
\begin{center}

\end{center}
\end{minipage}& \begin{minipage}[m]{.10\linewidth}\setlength{\unitlength}{.25mm}
\begin{center}

\end{center}
\end{minipage} & \begin{minipage}[m]{.10\linewidth}\setlength{\unitlength}{.25mm}
\begin{center}
56
\end{center}
\end{minipage}\\

\hline

\begin{minipage}[m]{.40\linewidth}\setlength{\unitlength}{.25mm}
\begin{center}
(1 ; 2, -1, -1, 0, 0, 0)
\end{center}
\end{minipage}
& \begin{minipage}[m]{.10\linewidth}\setlength{\unitlength}{.25mm}
\begin{center}
60
\end{center}
\end{minipage}& \begin{minipage}[m]{.10\linewidth}\setlength{\unitlength}{.25mm}
\begin{center}
1
\end{center}
\end{minipage} & \begin{minipage}[m]{.10\linewidth}\setlength{\unitlength}{.25mm}
\begin{center}
60
\end{center}
\end{minipage}\\

\hline

\begin{minipage}[m]{.40\linewidth}\setlength{\unitlength}{.25mm}
\begin{center}
(1 ; 1, 1, -1, -1, 0, 0)
\end{center}
\end{minipage}
& \begin{minipage}[m]{.10\linewidth}\setlength{\unitlength}{.25mm}
\begin{center}
90
\end{center}
\end{minipage}& \begin{minipage}[m]{.10\linewidth}\setlength{\unitlength}{.25mm}
\begin{center}
1
\end{center}
\end{minipage} & \begin{minipage}[m]{.10\linewidth}\setlength{\unitlength}{.25mm}
\begin{center}
90
\end{center}
\end{minipage}\\

\hline

\begin{minipage}[m]{.40\linewidth}\setlength{\unitlength}{.25mm}
\begin{center}
(1 ; 1, -1, 0, 0, 0, 0)
\end{center}
\end{minipage}
& \begin{minipage}[m]{.10\linewidth}\setlength{\unitlength}{.25mm}
\begin{center}
30
\end{center}
\end{minipage}& \begin{minipage}[m]{.10\linewidth}\setlength{\unitlength}{.25mm}
\begin{center}
2
\end{center}
\end{minipage} & \begin{minipage}[m]{.10\linewidth}\setlength{\unitlength}{.25mm}
\begin{center}
60
\end{center}
\end{minipage}\\

\hline

\begin{minipage}[m]{.40\linewidth}\setlength{\unitlength}{.25mm}
\begin{center}
(1 ; 0, 0, 0, 0, 0, 0)
\end{center}
\end{minipage}
& \begin{minipage}[m]{.10\linewidth}\setlength{\unitlength}{.25mm}
\begin{center}
1
\end{center}
\end{minipage}& \begin{minipage}[m]{.10\linewidth}\setlength{\unitlength}{.25mm}
\begin{center}
3
\end{center}
\end{minipage} & \begin{minipage}[m]{.10\linewidth}\setlength{\unitlength}{.25mm}
\begin{center}
3
\end{center}
\end{minipage}\\

\hline

\begin{minipage}[m]{.40\linewidth}\setlength{\unitlength}{.25mm}
\begin{center}

\end{center}
\end{minipage}
& \begin{minipage}[m]{.10\linewidth}\setlength{\unitlength}{.25mm}
\begin{center}

\end{center}
\end{minipage}& \begin{minipage}[m]{.10\linewidth}\setlength{\unitlength}{.25mm}
\begin{center}

\end{center}
\end{minipage} & \begin{minipage}[m]{.10\linewidth}\setlength{\unitlength}{.25mm}
\begin{center}
213
\end{center}
\end{minipage}\\

\hline

\begin{minipage}[m]{.40\linewidth}\setlength{\unitlength}{.25mm}
\begin{center}
(2 ; 2, -1, -1, -1, -1, -1)
\end{center}
\end{minipage}
& \begin{minipage}[m]{.10\linewidth}\setlength{\unitlength}{.25mm}
\begin{center}
6
\end{center}
\end{minipage}& \begin{minipage}[m]{.10\linewidth}\setlength{\unitlength}{.25mm}
\begin{center}
1
\end{center}
\end{minipage} & \begin{minipage}[m]{.10\linewidth}\setlength{\unitlength}{.25mm}
\begin{center}
6
\end{center}
\end{minipage}\\

\hline

\begin{minipage}[m]{.40\linewidth}\setlength{\unitlength}{.25mm}
\begin{center}
(2 ; 1, -1, -1, -1, -1, 0)
\end{center}
\end{minipage}
& \begin{minipage}[m]{.10\linewidth}\setlength{\unitlength}{.25mm}
\begin{center}
30
\end{center}
\end{minipage}& \begin{minipage}[m]{.10\linewidth}\setlength{\unitlength}{.25mm}
\begin{center}
2
\end{center}
\end{minipage} & \begin{minipage}[m]{.10\linewidth}\setlength{\unitlength}{.25mm}
\begin{center}
60
\end{center}
\end{minipage}\\
\hline
\begin{minipage}[m]{.40\linewidth}\setlength{\unitlength}{.25mm}
\begin{center}
(2 ; 1, -1, -1, -2, 0, 0)
\end{center}
\end{minipage}
& \begin{minipage}[m]{.10\linewidth}\setlength{\unitlength}{.25mm}
\begin{center}
180
\end{center}
\end{minipage}& \begin{minipage}[m]{.10\linewidth}\setlength{\unitlength}{.25mm}
\begin{center}
1
\end{center}
\end{minipage} & \begin{minipage}[m]{.10\linewidth}\setlength{\unitlength}{.25mm}
\begin{center}
180
\end{center}
\end{minipage}\\

\hline

\begin{minipage}[m]{.40\linewidth}\setlength{\unitlength}{.25mm}
\begin{center}
(2 ; 1, -2, -2, 0, 0, 0)
\end{center}
\end{minipage}
& \begin{minipage}[m]{.10\linewidth}\setlength{\unitlength}{.25mm}
\begin{center}
60
\end{center}
\end{minipage}& \begin{minipage}[m]{.10\linewidth}\setlength{\unitlength}{.25mm}
\begin{center}
1
\end{center}
\end{minipage} & \begin{minipage}[m]{.10\linewidth}\setlength{\unitlength}{.25mm}
\begin{center}
60
\end{center}
\end{minipage}\\

\hline

\begin{minipage}[m]{.40\linewidth}\setlength{\unitlength}{.25mm}
\begin{center}
(2 ; -1, -1, -1, 0, 0, 0)
\end{center}
\end{minipage}
& \begin{minipage}[m]{.10\linewidth}\setlength{\unitlength}{.25mm}
\begin{center}
20
\end{center}
\end{minipage}& \begin{minipage}[m]{.10\linewidth}\setlength{\unitlength}{.25mm}
\begin{center}
3
\end{center}
\end{minipage} & \begin{minipage}[m]{.10\linewidth}\setlength{\unitlength}{.25mm}
\begin{center}
60
\end{center}
\end{minipage}\\

\hline

\begin{minipage}[m]{.40\linewidth}\setlength{\unitlength}{.25mm}
\begin{center}
(2 ; -1, -2, 0, 0, 0, 0)
\end{center}
\end{minipage}
& \begin{minipage}[m]{.10\linewidth}\setlength{\unitlength}{.25mm}
\begin{center}
30
\end{center}
\end{minipage}& \begin{minipage}[m]{.10\linewidth}\setlength{\unitlength}{.25mm}
\begin{center}
2
\end{center}
\end{minipage} & \begin{minipage}[m]{.10\linewidth}\setlength{\unitlength}{.25mm}
\begin{center}
60
\end{center}
\end{minipage}\\

\hline

\begin{minipage}[m]{.40\linewidth}\setlength{\unitlength}{.25mm}
\begin{center}

\end{center}
\end{minipage}
& \begin{minipage}[m]{.10\linewidth}\setlength{\unitlength}{.25mm}
\begin{center}

\end{center}
\end{minipage}& \begin{minipage}[m]{.10\linewidth}\setlength{\unitlength}{.25mm}
\begin{center}

\end{center}
\end{minipage} & \begin{minipage}[m]{.10\linewidth}\setlength{\unitlength}{.25mm}
\begin{center}
426
\end{center}
\end{minipage}\\

\hline

\begin{minipage}[m]{.40\linewidth}\setlength{\unitlength}{.25mm}
\begin{center}
(3 ; 1, -1, -1, -1, -2, -2)
\end{center}
\end{minipage}
& \begin{minipage}[m]{.10\linewidth}\setlength{\unitlength}{.25mm}
\begin{center}
60
\end{center}
\end{minipage}& \begin{minipage}[m]{.10\linewidth}\setlength{\unitlength}{.25mm}
\begin{center}
1
\end{center}
\end{minipage} & \begin{minipage}[m]{.10\linewidth}\setlength{\unitlength}{.25mm}
\begin{center}
60
\end{center}
\end{minipage}\\

\hline

\begin{minipage}[m]{.40\linewidth}\setlength{\unitlength}{.25mm}
\begin{center}
(3 ; -1, -1, -1, -1, -1, -1)
\end{center}
\end{minipage}
& \begin{minipage}[m]{.10\linewidth}\setlength{\unitlength}{.25mm}
\begin{center}
1
\end{center}
\end{minipage}& \begin{minipage}[m]{.10\linewidth}\setlength{\unitlength}{.25mm}
\begin{center}
4
\end{center}
\end{minipage} & \begin{minipage}[m]{.10\linewidth}\setlength{\unitlength}{.25mm}
\begin{center}
4
\end{center}
\end{minipage}\\

\hline

\begin{minipage}[m]{.40\linewidth}\setlength{\unitlength}{.25mm}
\begin{center}
(3 ; -1, -1, -1, -1, -2,  0)
\end{center}
\end{minipage}
& \begin{minipage}[m]{.10\linewidth}\setlength{\unitlength}{.25mm}
\begin{center}
30
\end{center}
\end{minipage}& \begin{minipage}[m]{.10\linewidth}\setlength{\unitlength}{.25mm}
\begin{center}
3
\end{center}
\end{minipage} & \begin{minipage}[m]{.10\linewidth}\setlength{\unitlength}{.25mm}
\begin{center}
90
\end{center}
\end{minipage}\\

\hline

\begin{minipage}[m]{.40\linewidth}\setlength{\unitlength}{.25mm}
\begin{center}
(3 ; -1, -1, -1, -3, 0,  0)
\end{center}
\end{minipage}
& \begin{minipage}[m]{.10\linewidth}\setlength{\unitlength}{.25mm}
\begin{center}
60
\end{center}
\end{minipage}& \begin{minipage}[m]{.10\linewidth}\setlength{\unitlength}{.25mm}
\begin{center}
1
\end{center}
\end{minipage} & \begin{minipage}[m]{.10\linewidth}\setlength{\unitlength}{.25mm}
\begin{center}
60
\end{center}
\end{minipage}\\

\hline

\begin{minipage}[m]{.40\linewidth}\setlength{\unitlength}{.25mm}
\begin{center}
(3 ; -1, -1, -2, -2, 0,  0)
\end{center}
\end{minipage}
& \begin{minipage}[m]{.10\linewidth}\setlength{\unitlength}{.25mm}
\begin{center}
90
\end{center}
\end{minipage}& \begin{minipage}[m]{.10\linewidth}\setlength{\unitlength}{.25mm}
\begin{center}
2
\end{center}
\end{minipage} & \begin{minipage}[m]{.10\linewidth}\setlength{\unitlength}{.25mm}
\begin{center}
180
\end{center}
\end{minipage}\\

\hline

\begin{minipage}[m]{.40\linewidth}\setlength{\unitlength}{.25mm}
\begin{center}
(3 ; -1, -2, -3,  0, 0,  0)
\end{center}
\end{minipage}
& \begin{minipage}[m]{.10\linewidth}\setlength{\unitlength}{.25mm}
\begin{center}
120
\end{center}
\end{minipage}& \begin{minipage}[m]{.10\linewidth}\setlength{\unitlength}{.25mm}
\begin{center}
1
\end{center}
\end{minipage} & \begin{minipage}[m]{.10\linewidth}\setlength{\unitlength}{.25mm}
\begin{center}
120
\end{center}
\end{minipage}\\

\hline

\begin{minipage}[m]{.40\linewidth}\setlength{\unitlength}{.25mm}
\begin{center}
(3 ; -2, -2, -2,  0, 0,  0)
\end{center}
\end{minipage}
& \begin{minipage}[m]{.10\linewidth}\setlength{\unitlength}{.25mm}
\begin{center}
20
\end{center}
\end{minipage}& \begin{minipage}[m]{.10\linewidth}\setlength{\unitlength}{.25mm}
\begin{center}
1
\end{center}
\end{minipage} & \begin{minipage}[m]{.10\linewidth}\setlength{\unitlength}{.25mm}
\begin{center}
20
\end{center}
\end{minipage}\\

\hline

\begin{minipage}[m]{.40\linewidth}\setlength{\unitlength}{.25mm}
\begin{center}
(3 ; -3, -3, 0,  0, 0,  0)
\end{center}
\end{minipage}
& \begin{minipage}[m]{.10\linewidth}\setlength{\unitlength}{.25mm}
\begin{center}
15
\end{center}
\end{minipage}& \begin{minipage}[m]{.10\linewidth}\setlength{\unitlength}{.25mm}
\begin{center}
1
\end{center}
\end{minipage} & \begin{minipage}[m]{.10\linewidth}\setlength{\unitlength}{.25mm}
\begin{center}
15
\end{center}
\end{minipage}\\

\hline

\begin{minipage}[m]{.40\linewidth}\setlength{\unitlength}{.25mm}
\begin{center}

\end{center}
\end{minipage}
& \begin{minipage}[m]{.10\linewidth}\setlength{\unitlength}{.25mm}
\begin{center}

\end{center}
\end{minipage}& \begin{minipage}[m]{.10\linewidth}\setlength{\unitlength}{.25mm}
\begin{center}

\end{center}
\end{minipage} & \begin{minipage}[m]{.10\linewidth}\setlength{\unitlength}{.25mm}
\begin{center}
549
\end{center}
\end{minipage}\\

\hline

\begin{minipage}[m]{.40\linewidth}\setlength{\unitlength}{.25mm}
\begin{center}
(4 ; $\ast$, $\ast$, $\ast$, $\ast$, $\ast$, $\ast$)
\end{center}
\end{minipage}
& \begin{minipage}[m]{.10\linewidth}\setlength{\unitlength}{.25mm}
\begin{center}

\end{center}
\end{minipage}& \begin{minipage}[m]{.10\linewidth}\setlength{\unitlength}{.25mm}
\begin{center}

\end{center}
\end{minipage} & \begin{minipage}[m]{.10\linewidth}\setlength{\unitlength}{.25mm}
\begin{center}
426
\end{center}
\end{minipage}\\
\hline

\begin{minipage}[m]{.40\linewidth}\setlength{\unitlength}{.25mm}
\begin{center}
(5 ; $\ast$, $\ast$, $\ast$, $\ast$, $\ast$, $\ast$)
\end{center}
\end{minipage}
& \begin{minipage}[m]{.10\linewidth}\setlength{\unitlength}{.25mm}
\begin{center}

\end{center}
\end{minipage}& \begin{minipage}[m]{.10\linewidth}\setlength{\unitlength}{.25mm}
\begin{center}

\end{center}
\end{minipage} & \begin{minipage}[m]{.10\linewidth}\setlength{\unitlength}{.25mm}
\begin{center}
213
\end{center}
\end{minipage}\\
\hline
\begin{minipage}[m]{.40\linewidth}\setlength{\unitlength}{.25mm}
\begin{center}
(6 ; $\ast$, $\ast$, $\ast$, $\ast$, $\ast$, $\ast$)
\end{center}
\end{minipage}
& \begin{minipage}[m]{.10\linewidth}\setlength{\unitlength}{.25mm}
\begin{center}

\end{center}
\end{minipage}& \begin{minipage}[m]{.10\linewidth}\setlength{\unitlength}{.25mm}
\begin{center}

\end{center}
\end{minipage} & \begin{minipage}[m]{.10\linewidth}\setlength{\unitlength}{.25mm}
\begin{center}
56
\end{center}
\end{minipage}\\
\hline

%\begin{minipage}[m]{.40\linewidth}\setlength{\unitlength}{.25mm}
%\begin{center}

%\end{center}
%\end{minipage}
%& \begin{minipage}[m]{.10\linewidth}\setlength{\unitlength}{.25mm}
%\begin{center}

%\end{center}
%\end{minipage}& \begin{minipage}[m]{.10\linewidth}\setlength{\unitlength}{.25mm}
%\begin{center}

%\end{center}
%\end{minipage} & \begin{minipage}[m]{.10\linewidth}\setlength{\unitlength}{.25mm}
%\begin{center}
%1939
%\end{center}
%\end{minipage}\\

\end{longtable}
%%%%%%%%%%%%%%%%%%%

\subsubsection{$r=7,8$} By the same way, we can also compute the values $h^0(X_r, \mathcal{O}_{X_r}(i))$. Since the list is too long, we omit it.

%There are 5109841  effective divisors $D$ up to linear equivalence on $S_8$ with $\deg D=5$.

In summary, we have obtained the following table.

\noindent\[
\begin{array}{c|c|c|c}
r & d_r & g_r & h^0(X_r, \mathcal{O}_{X_r}(1)), \ldots, h^0(X_r, \mathcal{O}_{X_r}(r-3))  \\
\hline
4 & 5 & 1& 10 \\
\hline
5 & 34  & 35 & 16, 116   \\
\hline
6 & 372 & 745 & 27, 297, 1939   \\
\hline
7 & 9504 & 28513 & 56, 1067, 10576, 67949   \\
\hline
8 &  1779840 & 7119361 &  242, 12004, 226327, 2301371, 15449296  \\
\end{array}
\]

From this table, we can calculate the Hilbert functions $H_{\Cox(S_r)}(t)$ for $4 \leq r \leq 8$.
This proves Theorem \ref{hilftthm}.

As we have discussed in Subsection \ref{syzsubsec}, we can also compute $B_j(\Cox(S_r))=\sum_{i \geq 0} (-1)^i b_{i,j}(\Cox(S_r))$ from the Hilbert functions $H_{\Cox(S_r)}(t)$ for $4 \leq r \leq 8$. Here, we list $B_j(\Cox(S_r))$ for $r=4,5$.

\begin{corollary}\label{Bcox}
We have the following.
\noindent\[
\begin{array}{c|c|c|c|c|c|c|c|c|c|c|c|c|c}
j &                      0 & 1 & 2 & 3      & 4  &  5      & 6     & 7     & 8    & 9    & 10  & 11  & 12 \\
\hline
B_j(\Cox(S_4)) & 1 & 0 & -5   & 5  & 0  & -1     & 0     & 0      & 0   & 0   & 0    & 0  &0\\
\hline
B_j(\Cox(S_5)) & 1 & 0 & -20 & 48 & 7 & -176 & 280 & -176 & 7 & 48 & -20 & 0 & 1

\end{array}
\]

\end{corollary}

\section{Batyrev-Popov conjecture and Green-Lazarsfeld index}\label{appsec}
This section is devoted to some applications of the computation of some syzygetic invariants of Cox rings of del Pezzo surfaces. We compute the Green-Lazarsfeld index of Cox rings of del Pezzo surfaces.

\subsection{Batyrev-Popov conjecture}\label{bpsubsec}
In this subsection, we explain the Batyrev-Popov conjecture. Let $S_r$ be a del Pezzo surface of degree $9-r$. We assume that $4 \leq r \leq 8$.
We have the following exact sequence
$$
0 \longrightarrow I_r \longrightarrow k[G_r] \longrightarrow \Cox(S_r) \longrightarrow 0
$$
where $G_r$ is the set of generatros of $\Cox(S_r)$ (see Subsection \ref{coxsubsec}).

We now briefly explain how to find quadric generators of $I_r$.
For simplicity, we assume that $r=4$ for a while.
There are ten $(-1)$-curves on $S_4$: $E_i~~(1 \leq i \leq 4)$, $L-E_j-E_k~~(1 \leq j<k \leq 4)$ (see Subsection \ref{delsubsec}). We denote by $x_i$ and $x_{jk}$ their generating sections.
Then $k[G_4]=k[x_1, \ldots, x_4, x_{12}, \ldots, x_{34}]$.
Consider a conic $Q$ (say $L-E_1$) on $S_4$. It induces a fibration $f \colon S_4 \to \P^1$ with 3 singular fibers
$$
(L-E_1-E_2)+E_2, (L-E_1-E_3)+E_3, \text{and } (L-E_1-E_4)+E_4.
$$
We get a surjective linear map $f_Q \colon k \langle x_{12}x_2, x_{13}x_3, x_{14}x_4 \rangle \to H^0(\mathcal{O}_{S_4}(Q))$.
Note that $E_1=\P^1$ is a section of this fibration. We can regard $H^0(\mathcal{O}_{S_4}(Q)) \simeq H^0(\mathcal{O}_{E_1}(1))$.
We may assume 3 singular fibers meet $E_1$ at $[0:1],[1:0],[1,1]$, and then we obtain
$$
\ker(f_Q) = k \cdot  \langle x_{12}x_2-x_{13}x_3-x_{14}x_4 \rangle.
$$
We have seen that the conic $Q=L-E_1$ gives a quadric generator $x_{12}x_2-x_{13}x_3-x_{14}x_4$ of $I_4$.

In general, for a conic $Q$ on $S_r$ with $4 \leq r \leq 8$, the induced fibration $f \colon S_r \to \P^1$ has $r-1$ singular fibers. Thus $\ker(f_Q)$ is generated by $r-3$ quadrics so that each conic gives $r-3$ quadric generators of $I_r$.
When $r \geq 7$, we have further quadric generators of $I_r$.
On $S_7$, there are 25 additional quadric generators from$-K_{S_7}$. On $S_8$, there are 27 quadric generators from each $-K_{S_8}+E$ where $E$ is a $(-1)$-curve and 119 additional quadric generators from $-2K_{S_8}$. For more details, see \cite{D} and \cite{TVAV}.

The following was conjectured by Batyrev-Popov, and finally proved in \cite{TVAV} and \cite{SX}.

\begin{theorem}[{\cite{TVAV} and \cite{SX}}]\label{bp}
$I_r$ is generated by quadrics.
\end{theorem}

The number of minimal generators of $I_r$ is given as follows.
\noindent\[
\begin{array}{c|c|c|c|c|c}
r & 4 & 5 & 6 & 7 & 8 \\
\hline
\text{number of minimal generators of $I_r$} & 5 & 20 & 81 & 529 & 17399 \\
\end{array}
\]

\subsection{Green-Lazarsfeld index}\label{glsubsec}
The aim of this subsection is to prove the following theorem.

\begin{theorem}\label{index}
$\index(\Cox(S_4))=2$ and $\index(\Cox(S_r))=1$ for $5 \leq r \leq 8$.
\end{theorem}

\begin{proof}
We first deal with the case $r=4$.
Recall that $X_4 = Gr(2,5) \subset \P^9$. Then it is well known that $Gr(2,5)$ satisfies $N_2$ property but not $N_3$ property (see e.g., \cite[Proposition 3.8 and Remark 3.9]{EGHP}). Thus $\index(\Cox(S_4))=2$.
Considering the case $r=5$. We know that $\index(\Cox(S_5)) \geq 1$ by Proposition \ref{X_rprop} and  Theorem \ref{bp}. By Corollary \ref{Bcox}, we have
$$
7=B_4(\Cox(S_5))=-b_{3,4}(\Cox(S_5))+b_{2,4}(\Cox(S_5))
$$
so that $b_{2,4}(\Cox(S_5)) \geq 7$. Thus $\index(\Cox(S_5)) \leq 1$, and hence, $\index(\Cox(S_5))=1$. Finally, we consider the case $6 \leq r \leq 8$.
First, we can think that $I_5 \subset I_{r}$ for $6 \leq r \leq 8$. We have
$$
b_{2,4}(\Cox(S_{r})) \geq b_{2,4}(\Cox(S_5))
$$
for $6 \leq r \leq 8$
so that $b_{2,4}(\Cox(S_r)) \neq 0$ for $6 \leq r \leq 8$. Thus $\index(\Cox(S_r)) \leq 1$, and hence, $\index(\Cox(S_r))=1$ for $5 \leq r \leq 8$.
\end{proof}

We will see in Section \ref{syzdeg4subsec} that $b_{2, Q}(\Cox(S_r)) \neq 0$ for any conic $Q$ and $5 \leq r \leq 8$ where $b_{2,Q}$ denotes the multigraded Betti number (see Remark \ref{b2q}).

\section{Syzygies of Cox rings of del Pezzo surfaces}\label{syzdeg4subsec}
In this section, we completely determine the $\Pic(S_r)$-graded minimal free resolution of $\Cox(S_r)$ for $r=4,5$. For this purpose, we first briefly show some basic properties of multigraded Betti numbers of $\Cox(S_r)$ in full generality, and then, we compute the $\Pic(S_r)$-graded Betti numbers of $\Cox(S_r)$ for $r=4,5$.

\subsection{Basic properties of $\Pic(S_r)$-graded Betti numbers of $\Cox(S_r)$}

Let $S_r$ be a del Pezzo surface of degree $9-r$. As before, we assume that $4 \leq r \leq 8$.
We first briefly review the definitions of $\Pic(S_r)$-graded Betti numbers of $\Cox(S_r)$.
Let $R_r:=k[G_r]$.
Recall that $R_r$ and $\Cox(S_r)$ are $\Pic(S_r)$-graded (so is $I_r$).
The \emph{$\Pic(S_r)$-graded minimal free resolution} of $\Cox(S_r)$ is of the form
$$
 R_r \longleftarrow \bigoplus_{D \in \Pic(S_r)} R_r(-D)^{b_{1,D}(\Cox(S_r))} \longleftarrow \bigoplus_{D \in \Pic(S_r)} R_r(-D)^{b_{2,D}(\Cox(S_r))} \longleftarrow \cdots
$$
where $b_{i,D}(\Cox(S_r))$ denotes the \emph{$\Pic(S_r)$-graded Betti number}.
Let $C_1, \ldots, C_{|G_r|}$ be divisors of sections of $G_r$.
By considering the Koszul complex, the $\Pic(S_r)$-graded Betti number $b_{i,D}(\Cox(S_r))$ can be computed as the dimension of the homology of
$A(D)_{i+1} \to A(D)_i \to A(D)_{i-1}$ where
$$
A(D)_d:=\bigoplus_{1 \leq i_1 < \cdots < i_{d} \leq |G_r|} H^0(S_r, D-\sum C_{i_j}) .
$$
We refer to \cite{LV} and \cite{TVAV} for more details on $\Pic(S_r)$-graded Betti numbers.

As in the $\Z$-graded case, we also define $B_{D}(\Cox(S_r)):=\sum_{i \geq 0} (-1)^i b_{i,D}(\Cox(S_r))$. Then we have
$$
b_{i,j}(\Cox(S_r))=\sum_{\deg D=j} b_{i, D}(\Cox(S_r)) \text{ and } B_j(\Cox(S_r))=\sum_{\deg D=j} B_D(\Cox(S_r)).
$$

We now show some basic properties of $\Pic(S_r)$-graded Betti numbers of $\Cox(S_r)$.
First, we prove the nefness of some divisor $D$ with nonvanishing Betti number.

\begin{lemma}[Nefness]\label{nef}
If $b_{i,D}(\Cox(S_r)) \neq 0$ and $\deg D = -K_{S_r}.D=i+1$, then $D$ is nef.
\end{lemma}

\begin{proof}
We use the induction on $i$.
If $b_{1,D}(\Cox(S_r)) \neq 0$ and $\deg D=2$, then $D$ is nef as we already saw in Subsection \ref{bpsubsec}.
Alternatively, one can directly prove it (see \cite[Lemma 4.1]{LV}).
For the induction, we now assume that the assertion holds for $i<k$ and $b_{k,D}(\Cox(S_r)) \neq 0$ for some divisor $D$ with $\deg D=k+1$.
Then we have $b_{k-1,D-E}(\Cox(S_r)) \neq 0$ for some $(-1)$-curve $E$ on $S_r$. By the induction, $D':=D-E$ is nef and $\deg D'=k$.
Suppose that $D$ is not nef. Then we must have $D.E<0$.
Note that there is no nef divisor $D''$ on $S_r$ different from $D'$ such that $D=D''+E'$ where $E'$ is a $(-1)$-curve on $S_r$.
Thus $R_r(-D)^{b_{k,D}(\Cox(S_r))}$ maps into $R_r(-D')^{b_{k-1,D'}(\Cox(S_r))}$ in the minimal free resolution of $\Cox(S_r)$. 
Now consider the map $R_r(-D')^{b_{k-1,D'}(\Cox(S_r))} \to \bigoplus_{B \in \Pic(S_r)} R_r(-B)^{b_{k-2, B}(\Cox(S_r))}$ in the minimal free resolution of $\Cox(S_r)$. 
Since $R_r(-D')_{D'}^{b_{k-1,D'}(\Cox(S_r))}$ injects into $\left(\bigoplus_{B \in \Pic(S_r)} R_r(-B)^{b_{k-2, B}(\Cox(S_r))} \right)_{D'}$, and $R_r(-D')_{D'}^{b_{k-1,D'}(\Cox(S_r))}$ and $R_r(-D')_{D}^{b_{k-1,D'}(\Cox(S_r))}$ are the same dimensional vector spaces, it follows that $R_r(-D')_{D}^{b_{k-1,D'}(\Cox(S_r))}$ also injects into  $\left(\bigoplus_{B \in \Pic(S_r)} R_r(-B)^{b_{k-2, B}(\Cox(S_r))} \right)_{D}$.
Then we have $R_r(-D)^{b_{k,D}(\Cox(S_r))}=0$, which is a contradiction.
\end{proof}

Recall that the Weyl group $W_r$ of the root system $R_r \subset \Pic(S_r)$ naturally acts on $\Pic(S_r)$.

\begin{lemma}[Weyl group invariance]\label{weyl}
If $D, D'  \in \Pic(S_r)$ are in the same orbit of $W_r$-action, then $b_{i,D}(\Cox(S_r))=b_{i,D'}(\Cox(S_r))$.
\end{lemma}

\begin{proof}
Recall that $b_{i,D}(\Cox(S_r))$ can be computed as the dimension of the homology of
$A(D)_{i+1} \to A(D)_i \to A(D)_{i-1}$ where
$$
A(D)_d:=\bigoplus_{1 \leq i_1 < \cdots < i_{d} \leq |G_r|} H^0(S_r, D-\sum C_{i_j}) .
$$
If $D$ and $D'$ are in the same orbit of $W_r$-action, then there is an isomorphism between $A(D)_d$ and $A(D')_d$. Thus we get $b_{i,D}(\Cox(S_r))=b_{i,D'}(\Cox(S_r))$.
\end{proof}

We also have the following duality (cf. Corollary \ref{greendual}).

\begin{lemma}[Duality]\label{dual}
$b_{i, D}(\Cox(S_r))=b_{\pd(\Cox(S_r))-i, \sum_{j=1}^{|G_r|}C_j + K_{S_r}-D }(\Cox(S_r))$.
\end{lemma}

\begin{proof}
Since $\Cox(S_r)$ is Gorenstein by Proposition \ref{popov}, the dual of the minimal free resolution of $\Cox(S_r)$ is again the minimal free resolution of $\omega_{\Cox(S_r)}\left(\sum_{j=1}^{|G_r|} C_j \right)$.
By \cite[Corollary 1.5]{HaKu}, we have $\omega_{\Cox(S_r)}\left(\sum_{j=1}^{|G_r|} C_j \right) = \Cox(S_r)\left (\sum_{j=1}^{|G_r|} C_j -K_{S_r} \right)$. Then the assertion immediately follows.
\end{proof}

\subsection{Syzygies of Cox rings of a del Pezzo surface of degree 5}
The aim of this subsection is to prove the following in several ways.

\begin{theorem}\label{bettiS_4}
The Betti diagram of $\Cox(S_4)$ is given as follows.
\noindent\[
\begin{array}{cccc}
1 & - & - & - \\
- & 5 & 5 & -\\
- & - & - & 1
\end{array}
\]
\end{theorem}

\begin{proof}
Since we know $X_4=Gr(2,5) \subset \P^9$, we can easily verify the assertion. However, we give an alternative proof. Recall that $\reg(\Cox(S_4))=2$ and $\pd(\Cox(S_4))=3$ (Corollary \ref{regpd}). Furthermore, $\Cox(S_4)$ is Gorenstein (Proposition \ref{popov}). Thus we only have to determine $b_{1,2}(\Cox(S_4))$ and $b_{2,3}(\Cox(S_4))$.
We now have
$$
-5=B_2(\Cox(S_4))=b_{1,2}(\Cox(S_4)) \text{ and } 5=B_3(\Cox(S_4))=b_{2,3}(\Cox(S_4)),
$$
so we get the assertion.
\end{proof}

Now we present the $\Pic(S_4)$-minimal free resolution of $\Cox(S_4)$.
As in Subsection \ref{bpsubsec}, we write $R_4=k[x_1, \ldots, x_4, x_{12}, \ldots, x_{34}]$.
Recall that $X_4=Gr(2,5) \subset \P^9$ is arithematically Gorenstein in codimension 3.
Using {Buchsbaum-Eisenbud theorem}, we obtain the $\Pic(S_4)$-minimal free resolution of $\Cox(S_4)$
$$
 R_4 \xleftarrow{\varphi} \bigoplus_{Q:\text{conic}} R_4(-Q) \xleftarrow{M} \bigoplus_{C:\text{twisted  cubic}} R_4(-C) \xleftarrow{\varphi^t} R(K_{S_4}) \leftarrow 0
$$
where entries of matrix $\varphi$ are {Pfaffians} of $M$ and
\begin{displaymath}
M=
\left( \begin{array}{ccccc}
0 & x_1 & x_2 & x_3 & x_4\\
-x_1 & 0 & x_{34} & x_{24} & x_{23}\\
-x_2 & -x_{34} & 0 & x_{14} & x_{13}\\
-x_3 & -x_{24} & -x_{14} & 0 & x_{12}\\
-x_4 & -x_{23} & -x_{13} & -x_{12} & 0
\end{array} \right).
\end{displaymath}
In particular, we have also shown Theorem \ref{bettiS_4}.

\subsection{Syzygies of Cox rings of a del Pezzo surface of degree 4}
The aim of this subsection is not only to prove the following but also to determine the $\Pic(S_5)$-graded Betti numbers of $\Cox(S_5)$.

\begin{theorem}\label{bettiS_5}
The Betti diagram of $\Cox(S_5)$ is given as follows.
\noindent\[
\begin{array}{ccccccccc}
1 & - & - & - & - & - & - & - & -\\
- & 20 & 48 & 3 & -   & -   & -   & -   &-\\
- & - & 10 & 176 & 280   & 176   & 10   & -   &-\\
- & - & -  & -     & -       & 3   & 48   & 20   &-\\
- & - & - & - & -   & -   & -   & -   &1
\end{array}
\]
\end{theorem}

\begin{proof}
The proof is divided into several steps. Our proof heavily depends on the computation of the Hilbert function of $\Cox(S_5)$, and we freely use Corollary \ref{Bcox}. For simplicity, we let $b_{i,j}:=b_{i,j}(\Cox(S_5))$.

\smallskip

Step 1. Recall that $\reg(\Cox(S_5))=4$ and $\pd(\Cox(S_5))=8$ (Corollary \ref{regpd}), and $\Cox(S_5)$ is Gorenstein (Proposition \ref{popov}). By duality (Corollary \ref{greendual} and Lemma \ref{dual}), we only have to consider the left half of the Betti diagram of $\Cox(S_5)$, which is given as below.
\noindent\[
\begin{array}{ccccc}
1 & - & - & - & - \\
- & 20 & b_{2,3} & b_{3,4} & b_{4,5}   \\
- & - & b_{2,4} & b_{3,5} & b_{4,6}  \\
- & - & b_{2,5}  & b_{3,6}  & b_{4,7}     \\
- & - & - & - & -
\end{array}
\]
Note that
$$b_{1,2}=-B_2=-\sum_{Q: \text{ conic}} B_{Q}=10 \times 2 =20.$$

\smallskip

Step 2. Now we determine $b_{2,3}$.
Every degree 3 nef divisor is a twisted cubic $C$, and we have
$$
b_{2,3}=B_3=\sum_{C:\text{ twisted cubic}} B_{C} = 16 \times 3=48.
$$

\smallskip

Step 3. Now we determine $b_{3,4}$ and $b_{2,4}$. This is the most difficult part.
Recall that $B_4=-b_{3,4}+b_{2,4}=7$.
There are three degree 4 nef divisors up to Weyl group action
$$
C+E \text{ with } C.E=1, ~~-K_{S_5},~~2Q
$$
where $Q$ is a conic, $C$ is a twisted cubic, and $E$ is a $(-1)$-curve. We also have $B_{C+E}=0, B_{-K_{S_5}}=3, B_{2Q}=1$.
Since $2Q-C$ is not effective so that $b_{3,2Q}=0$. Thus $b_{2,2Q}=1$.

Now we claim that $b_{2,C+E}=b_{3,C+E}=0$.
We can write $R_5=k[x_1, \ldots, x_5, y_{12}, \ldots, y_{45}, z]$ so that $\Cox(S_5)=R_5/I_5$ where $x_i$ corresponds to $E_i$, $y_{jk}$ corresponds to $L-E_j-E_k$, and $z$ corresponds to $2L-E_1-E_2-E_3-E_4-E_5$. Fix a monomial order on $R_5$ by giving weights as follows.
\noindent\[
\begin{array}{cccccccccc}
x_4 & y_{13} & x_2 & y_{23} & y_{15} & z & x_1 &y_{25} & y_{45} & \text{others}\\
13 & 11 & 10 & 9 & 8 & 8 & 7 & 7 &  1 & 6
\end{array}
\]
Using {\tt Macaulay2}, we can check that $b_{2,2L-E_1-E_2}(\text{in}(I_5))=0$, and hence, $b_{2,2L-E_1-E_2}=0$ by the upper semicontinuity (\cite[Theorem 8.29]{MS}).
Since any nef divisor of the form $C+E$ is in the same Weyl group orbit $W_5 \cdot (2L-E_1-E_2)$, it follows from Lemma \ref{weyl} that $b_{2,C+E}=0$. Since $B_{C+E}=0$, it follows that $b_{3,C+E}=0$. Here we remark that $b_{2,2L-E_1-E_3}(\text{in}(I_5)) \neq 0$.

Now we claim that $b_{2, -K_{S_5}}=0, b_{3, -K_{S_5}}=3$.
To show the claim, we use Gr\"{o}bner bases. For details, we refer to \cite{E2}.
Note that the 20 minimal quadratic generators of $I_5$ form a Gr\"{o}bner basis with respect to the previous monomial order.
By considering Buchberger algorithm and Schreyer theorem, if $b_{2,-K_{S_5}} \neq 0$, then the corresponding syzygy is of the form
$$
f g - g f = 0
$$
where $f \in I_{5,Q}$ for some conic $Q$ and $g \in I_{5,-K_{S_5}-Q}$. 
For any $(-1)$-curve $E$ such that $Q+E$ is a twisted cubic, there is a syzygy from $b_{2,3}$ containing $ef$ where $e \in H^0(S_5, \mathcal{O}_{S_5}(E))$ is a nonzero section. Then it follows that the syzygy of the form $fg-gf$ is generated by syzygies from $b_{2,3}$. 
For instance, we consider $f=-y_{12}x_2 + y_{13}x_3 + y_{14}x_4$ and $g=y_{25}y_{34}-y_{24}y_{35} + y_{23}y_{45}$. Then there are syzygies $s_1, \ldots, s_6$ from $b_{2,3}$ containing $y_{25}f, y_{24}f, y_{23}f, x_2g, x_3g, x_4g$, respectively. Then we have
$$
-y_{12}s_4 +y_{13}s_5+y_{14}s_6 -y_{34}s_1 + y_{35}s_2 - y_{45}s_3=fg-gf  + \ell, \text{ where }
$$
$$
\ell= -y_{12}(s_4-x_2g) +y_{13}(s_5-x_3g)+y_{14}(s_6-x_4g) -y_{34}(s_1-y_{25}f) + y_{35}(s_2-y_{24}f) - y_{45}(s_3-y_{23}f)
$$
Since $-y_{12}s_4 +y_{13}s_5+y_{14}s_6 -y_{34}s_1 + y_{35}s_2 -y_{45}s_3=fg-gf=0$, it follows that $\ell=0$. We can conclude that the syzygy of the form $fg-gf$ is generated by syzygies from $b_{2,3}$. Thus $b_{2,-K_{S_5}} =0$ so that $b_{3, -K_{S_5}}=3$.
Therefore, $b_{2,4}=\sum_{Q:conic} b_{2,2Q}=10 \times 1=10$ and $b_{3,4}=b_{3, -K_{S_5}}=3$.

\smallskip

Step 4. We now claim that $b_{i,i+1}=0$ for all $i \geq 4$.
Suppose that $b_{4,5} \neq 0$. Then we must have $b_{4,-K_{S_5}+E} \neq 0$ for any $(-1)$-curve $E$.
We have seen that $b_{3,4}=b_{3, -K_{S_5}}=3$.
By considering the minimal free resolution at $[-K_{S_5}+E]$, we obtain the following exact sequence
$$
\bigoplus_{C: \text{twisted cubic}} R_5(-C)_{-K_{S_5}+E}^3 \leftarrow R_5(K_{S_5})_{-K_{S_5}+E}^3 \leftarrow R_5(K_{S_5}-E)_{-K_{S_5}+E}^{b_{4,-K_{S_5}+E} }.
$$
Since $R_5(K_{S_5})_{-K_{S_5}}^3$ injects into $\bigoplus_{C: \text{twisted cubic}} R_5(-C)_{-K_{S_5}}^3$ and both $R_5(K_{S_5})_{-K_{S_5}}^3$ and $R_5(K_{S_5})_{-K_{S_5}+E}^3$ are three dimensional vector spaces, 
 $R_5(K_{S_5})_{-K_{S_5}+E}^3$ also injects into $\bigoplus_{C: \text{twisted cubic}} R_5(-C)_{-K_{S_5}+E}^3$.
Thus $R_5(K_{S_5}-E)_{-K_{S_5}+E}^{b_{4,-K_{S_5}+E} }=0$, and hence, $b_{4,5} =0$.
This shows the claim. By the dualtiy (Corollary \ref{greendual}), we obtain $b_{2,5}=b_{3,6}=b_{4,7}=0$.

\smallskip

Step 5. The remaining part is immediate. 
Recall that $B_5=-176$ and $B_6=280$. We have shown in the previous steps that $B_5=-b_{3,5}$ and $B_6=b_{4,6}$. Thus we get 
$b_{3,5}=176$ and  $b_{4,6}=280$.
\end{proof}

The left half of the $\Pic(S_5)$-graded Betti numbers of $\Cox(S_5)$ is given as follows.
\[
\begin{array}{c|c|c|c}
 Q \times 10 & C \times 16 & -K_{S_5} \times 1 & \\
  b_{1,Q}=2   & b_{2,C}=3     &  b_{3,-K_{S_5}}=3 &   \\
\hline
  & 2Q \times 10 & C+Q (C.Q=1) \times 80 & C+C' (C.C'=3) \times 80\\
   & b_{2,2Q}=1 & b_{3,C+Q}=1 & b_{4, C+C'}=2\\
   &                  & C+Q (C.Q=2) \times 16 & C+C' (C.C'=2) \times 10\\
   &                    & b_{3,C+Q}=6 & b_{4,C+C'}=12
\end{array}
\]
By the multigraded version duality (Lemma \ref{dual}), we can easily compute the remaining $\Pic(S_5)$-graded Betti numbers of $\Cox(S_5)$.

\begin{remark}\label{b2q}
We saw that $b_{2,Q}(\Cox(S_5)) \neq 0$ for any conic $Q$ on $S_5$. 
Recall that $I_5 \subset I_{r}$ for $6 \leq r \leq 8$.
Since $Q-E$ is not nef for any $(-1)$-curve $E$, it follows from Lemma \ref{nef} that 
$b_{2,Q}(\Cox(S_r)) \neq 0$ for any conic $Q$ on $S_r$ for $5 \leq r \leq 8$.
\end{remark}

\section*{Acknowledgements}
The authors are grateful to Bernd Sturmfels and Jae-Hyouk Lee for useful discussions.
The authors also wish to thank the referee for making a number of helpful suggestions.

$ $
%%%%%%%%%%%%%%%%%%%%%%%%%%%%%%%%%%%%%%%%%%%%%%%%%%%%%%%%%%%%%%%%%%%%%%%%%%%%%%%%%%%%%%%%%%%%%%%%%%%%%%%%
%BIBLIOGRAFIA

\end{document}